\documentclass[eqthmnum, nocolour]{jt-calcs}

\usepackage[backend=bibtex,style=alphabetic,sorting=nyt]{biblatex}
\bibliography{combined-bib}
\usepackage{verbatim}
\usepackage{mathrsfs}
\usepackage{mathscinet}

\setlist[enumerate,1]{label=(\alph*)}

\newcommand{\Q}{\mathbb{Q}}

\newcommand{\aut}{\mathrm{Aut}}
\newcommand{\syl}{\mathrm{Syl}}
\newcommand{\irr}{\mathrm{Irr}}
\newcommand{\gal}{\mathrm{Gal}}

\newcommand{\wt}[1]{\widetilde{#1}}

\newcommand{\PSU}{\mathrm{PSU}}

\newcommand{\GU}{\ensuremath{\mathrm{GU}}}

\renewcommand{\epsilon}{\varepsilon}

\addressm{MSU Denver, PO Box 173362, Campus Box 38, Denver, CO 80217-3362, USA}
\emailm{aschaef6@msudenver.edu}

\addresst{University of Arizona, 617 N. Santa Rita Ave., Tucson AZ 85721, USA}
\emailt{jaytaylor@math.arizona.edu}

\title{On Self-Normalising Sylow $2$-Subgroups in Type $\A$}
\author{Amanda Schaeffer Fry and Jay Taylor}

\begin{document}
\begin{abstract}
Navarro has conjectured a necessary and sufficient condition for a finite group $G$ to have a self-normalising Sylow $2$-subgroup, which is given in terms of the ordinary irreducible characters of $G$. The first-named author has reduced the proof of this conjecture to showing that certain related statements hold when $G$ is quasisimple. In this article we show that these conditions are satisfied when $G/Z(G)$ is $\PSL_n(q)$, $\PSU_n(q)$, or a simple group of Lie type defined over a finite field of characteristic $2$.
\end{abstract}

\section{Introduction}

\begin{pa}
For any integer $n \geqslant 1$ we will denote by $\Q_n$ the $n$th cyclotomic field, obtained from the rationals $\mathbb{Q}$ by adjoining a primitive $n$th root of unity. In \cite{schaeffer-fry:2015:odd-degree-characters}, the first-named author began an investigation into the following conjecture.
\end{pa}

\begin{conj}[Navarro]\label{mainprob}
Let $G$ be a finite group and let $\sigma\in\gal(\Q_{|G|}/\Q)$ be an automorphism fixing $2$-roots of unity and squaring $2'$-roots of unity. Then $G$ has a self-normalising Sylow $2$-subgroup if and only if every ordinary irreducible character of $G$ with odd degree is fixed by $\sigma$.
\end{conj}

\begin{pa}
This statement would be an immediate consequence of the Galois-McKay conjecture, which is a refinement of the well-known McKay conjecture due to Navarro, see \cite[Conjecture A]{navarro:2004:galois-mckay}. For a finite group $G$ we denote by $\Irr(G)$ the set of ordinary irreducible characters and given a prime $\ell$ we denote by $\Irr_{\ell'}(G) \subseteq \Irr(G)$ those irreducible characters whose degree is coprime to $\ell$. The Galois-McKay conjecture then posits that for any finite group $G$, prime $\ell$, and Sylow $\ell$-subgroup $P \leqslant G$, there should exist a bijection between $\irr_{\ell'}(G)$ and $\irr_{\ell'}(N_G(P))$, as predicted by the McKay conjecture, which behaves nicely with respect to the action of certain elements of the Galois group.
\end{pa}

\begin{pa}
While the McKay conjecture has been reduced to proving certain inductive statements for simple groups in \cite{imn:2007:reduction-mckay}, and even recently proven for $\ell=2$ in \cite{malle-spath:2015:mckay2}, no such reduction yet exists for Galois-McKay.  Further, a reduction for Galois-McKay seems further from fruition in the case $\ell=2$ than for odd primes.  However, a proof of \cref{mainprob} would provide more evidence for the conjecture, and we consider it to be a weak form of the Galois-McKay refinement for $\ell=2$.  We hope that some of the observations made in the course of proving \cref{mainprob} will be useful in working with an eventual reduction for Galois-McKay for $\ell=2$.  We also remark that the corresponding weak form for odd $\ell$ has been proven in  \cite{navarro-tiep-turull:2007:p-rational-characters}.
\end{pa}

\begin{pa}
The main result of \cite{schaeffer-fry:2015:odd-degree-characters} is a reduction of \cref{mainprob} to certain inductive statements for simple groups, which we recall below in \cref{sec:SimpleStatements}, and the verification of these statements for some simple groups. The goal of this work is to extend and simplify the proofs there in order to complete the verification for simple groups of Lie type in characteristic $2$ and simple groups of type $\A$ in all characteristics. Specifically we prove the following.
\end{pa}

\begin{maintheorem}\label{thm:SN2SgoodTypeA}
Assume $\bG$ is a simple and simply connected algebraic group defined over $\mathbb{K} = \overline{\mathbb{F}_p}$, an algebraic closure of the finite field $\mathbb{F}_p$ of prime order $p>0$, and let $F : \bG \to \bG$ be a Frobenius endomorphism of $\bG$. If either $p=2$ or $\bG = \SL_n(\mathbb{K})$, then whenever the quotient $\bG^F/Z(\bG^F)$ is simple, it is SN2S-Good.
\end{maintheorem}

\begin{pa}
One of our key tools used in the proof of \cref{thm:SN2SgoodTypeA} is Kawanaka's generalised Gelfand--Graev representations (GGGRs). These are a family of characters which have already shown themselves to be remarkably useful for deducing the action of automorphisms of a finite reductive group on the set of irreducible characters, see \cite{cabanes-spaeth:2015:equivariance-and-extendibility-in-groups-of-type-A} and \cite{taylor:2016:action-of-automorphisms-symplectic}. One of the reasons why they are so useful is that the image of a GGGR under an automorphism of the group is again a GGGR and the resulting GGGR can be easily described. Here we show that the same holds for certain Galois automorphisms, see \cref{prop:galois-aut-GGGR}. The statement holds whenever the GGGRs are defined and may be of independent interest.
\end{pa}

\begin{pa}
In \cite{schaeffer-fry:2015:odd-degree-characters} it is shown that all sporadic simple groups and simple alternating groups are SN2S-Good. Thus we are left with checking that most simple groups of Lie type defined over a field of odd characteristic are SN2S-Good. In this situation one should be able to employ the Harish-Chandra techniques used by Malle and Sp\"ath in \cite{malle-spath:2015:mckay2} to solve the McKay Conjecture for $\ell=2$. However, this is ultimately quite different from our line of argument here and will be considered elsewhere.
\end{pa}

\begin{pa}
We now outline the structure of the paper.  In \cref{sec:SimpleStatements}, we discuss the reduction of \cref{mainprob} to simple groups proved in \cite{schaeffer-fry:2015:odd-degree-characters}.  In \cref{sec:lusztig-series}, we introduce some general notation regarding finite reductive groups and the action of the Galois group on Lusztig series under specific conditions.  In \cref{sec:GGGRs}, we continue this discussion by introducing generalized Gelfand-Graev characters and their behavior under the action of the Galois group.  Sections \ref{sec:p2IF} and \ref{sec:condition-conjFI} are dedicated to proving \cref{thm:SN2SgoodTypeA} in the case that $p=2$.  In the remaining sections, we prove \cref{thm:SN2SgoodTypeA} for $\bG=\SL_n(\mathbb{K})$. 
\end{pa}

\section{The Reduction Statements for Simple Groups}\label{sec:SimpleStatements}
\begin{pa}
In \cite{schaeffer-fry:2015:odd-degree-characters} it was shown that \cref{mainprob} holds for any finite group if every finite simple group is SN2S-Good. The notion of being SN2S-Good is comprised of two conditions. One condition is on the simple group itself and the second is on its quasisimple covering groups. Before stating these conditions, we introduce some notation.
\end{pa}

\begin{pa}[Notation]
Let $G$ be a finite group. We will denote by $\Aut(G)$ the automorphism group of $G$. If $Q \leqslant \Aut(G)$ is any subgroup then we denote by $GQ$ the semidirect product of $Q$ acting on $G$. As in the introduction, $\Irr(G)$ denotes the set of ordinary irreducible characters of $G$ and $\Irr_{\ell'}(G) \subseteq \Irr(G)$ is the set of those irreducible characters whose degree is coprime to $\ell$, where $\ell$ is a prime. The set of all Sylow $\ell$-subgroups of $G$ will be denoted by $\syl_{\ell}(G)$. If $H \leqslant G$ is a subgroup of $G$ and $\chi \in \Irr(H)$ is an irreducible character then we denote by $\Irr(G|\chi)$ the set of all irreducible characters $\psi \in \Irr(G)$ whose restriction $\Res_H^G(\psi)$ to $H$ contains $\chi$ as an irreducible constituent; we say that $\psi$ covers $\chi$. Moreover, for any element $g \in G$ we denote by ${}^g\chi$ the irreducible character of ${}^gH = gHg^{-1}$ defined by ${}^g\chi(h) = \chi(g^{-1}hg)$ for all $h \in {}^gH$. We will write $\Irr_{\ell'}(G|\chi)$ for the intersection $\Irr(G|\chi) \cap \Irr_{\ell'}(G)$.
\end{pa}

\begin{assumption}
From this point forward $\sigma \in \Gal(\mathbb{Q}_{|G|}/\mathbb{Q})$ will denote the Galois automorphism fixing $2$-roots of unity and squaring $2'$-roots of unity, c.f., \cref{mainprob}.
\end{assumption}

\begin{condition}\label{cond:conjIF}
Let $G$ be a finite quasisimple group with centre $Z \leqslant G$ and $Q \leqslant \Aut(G)$ a $2$-group. Assume there exists a $Q$-invariant Sylow $2$-subgroup $P/Z\in \syl_2(G/Z)$ such that $C_{N_G(P)/P}(Q)=1$.  Then for any $Q$-invariant and $\sigma$-fixed $\lambda\in\irr(Z)$, we have $\chi^\sigma=\chi$ for any $Q$-invariant $\chi\in\irr_{2'}(G|\lambda)$.
\end{condition}

\begin{condition}\label{cond:conjFI}
Let $S$ be a finite nonabelian simple group and $Q \leqslant \Aut(S)$ a $2$-group. Assume $P \in \syl_2(S)$ is a $Q$-invariant Sylow then if every $Q$-invariant $\chi\in\irr_{2'}(S)$ is fixed by $\sigma$ we have $C_{N_S(P)/P}(Q)=1$.
\end{condition}

\begin{definition}\label{def:Goodness}
Let $S$ be a finite non-abelian simple group. We say $S$ is \emph{SN2S-Good} if \cref{cond:conjFI} holds for $S$ and \cref{cond:conjIF} holds for any quasisimple group $G$ satisfying $G/Z\cong S$.
\end{definition}

\begin{pa}\label{pa:conditions-remarks}
We end this section with some remarks concerning the above conditions. Firstly, if $P \leqslant G$ and $Q \leqslant \Aut(G)$ are as in \cref{cond:conjIF} then the condition that $C_{N_G(P)/P}(Q)=1$ is equivalent to $GQ/Z$ having a self-normalising Sylow $2$-subgroup, see \cite[Lemma 2.1(ii)]{navarro-tiep-turull:2007:p-rational-characters}. Secondly, assume $G$ is quasisimple with simple quotient $S = G/Z$ and let $\hat{G}$ be a universal perfect central extension, or Schur cover, of $S$. It is easily checked that if \cref{cond:conjIF} holds for $\hat{G}$ then it holds for $G$. Indeed, as $\hat{G}$ is a Schur cover there exists a surjective homomorphism $\hat{G} \to G$ with central kernel. This induces an injective map $\Irr(G) \to \Irr(\hat{G})$ and a surjective homomorphism $\Aut(\hat{G}) \to \Aut(G)$, see \cite[Corollary 5.1.4(a)]{gorenstein-lyons-solomon:1998:classification-3}, and the claim follows.
\end{pa}

\begin{rem}
We note that a simplified version of one side of the reduction, namely \cref{cond:conjIF}, has been proven in \cite{navarro-tiep:2015:irreducible-representations-of-odd-degree}. However, for the purposes of this paper, we work with our stronger condition.
\end{rem}

\section{Galois Automorphisms and Lusztig Series}\label{sec:lusztig-series}
\begin{assumption}
From this point forward we denote by $\mathbb{K} = \overline{\mathbb{F}_p}$ an algebraic closure of the finite field of prime order $p$. Moreover, $\ell$ denotes a prime.
\end{assumption}

\begin{pa}[The Basic Setup]\label{pa:basic-setup}
We introduce here the basic setup that will be used throughout this article. In particular, $\bG$ will be a connected reductive algebraic group defined over $\mathbb{K}$ and $F : \bG \to \bG$ will be a Frobenius endomorphism admitting an $\mathbb{F}_q$-rational structure $G = \bG^F$. Moreover, we denote by $\iota : \bG \hookrightarrow \widetilde{\bG}$ a regular embedding, in the sense of \cite[\S7]{lusztig:1988:reductive-groups-with-a-disconnected-centre}. The Frobenius endomorphism of $\widetilde{\bG}$ will again be denoted by $F$ and $\widetilde{G} = \widetilde{\bG}^F$ will be the resulting finite reductive group. 

We assume fixed pairs $(\bG^{\star},F^{\star})$ and $(\widetilde{\bG}^{\star},F^{\star})$ dual to $(\bG,F)$ and $(\widetilde{\bG},F)$ respectively. As before we set $G^{\star} = \bG^{\star F^{\star}}$ and $\widetilde{G}^{\star} = \widetilde{\bG}^{\star F^{\star}}$. We now choose an $F$-stable maximal torus $\bT_0 \leqslant \bG$ and a dual $F^{\star}$-stable maximal torus $\bT_0^{\star} \leqslant \bG^{\star}$. The group $\widetilde{\bT}_0 := \iota(\bT_0)Z(\widetilde{\bG})$ is then an $F$-stable maximal torus of $\widetilde{\bG}$. Recall that the regular embedding $\iota$ induces a surjective homomorphism $\iota^{\star} : \widetilde{\bG}^{\star} \to \bG^{\star}$ which is defined over $\mathbb{F}_q$. If $\widetilde{\bT}_0^{\star} \leqslant \widetilde{\bG}^{\star}$ is a torus dual to $\widetilde{\bT}_0$ then $\iota^{\star}(\widetilde{\bT}_0^{\star}) = \bT_0^{\star}$ and $\iota^{\star}$ is unique up to composing with an inner automorphism affected by an element of $\widetilde{\bT}_0^{\star}$.
\end{pa}

\begin{pa}\label{pa:dl-bijection}
We will denote by $\mathcal{C}(\bG,F)$ the set of all pairs $(\bT,\theta)$ consisting of an $F$-stable maximal torus $\bT \leqslant \bG$ and an irreducible character $\theta \in \Irr(\bT^F)$. Note we have an action of $G$ on $\nabla(\bG,F)$ defined by $g\cdot (\bT,\theta) = ({}^g\bT,{}^g\theta)$; we write $\mathcal{C}(\bG,F)/G$ for the orbits under this action and $[\bT,\theta]$ for the orbit containing $(\bT,\theta)$. Dually, we denote by $\mathcal{S}(\bG^\star,F^\star)$ the set of all pairs $(\bT^{\star},s)$ consisting of an $F^\star$-stable maximal torus $\bT^{\star} \leqslant \bG^{\star}$ and a semisimple element $s \in {\bT^{\star}}^{F^{\star}}$. Again we have an action of $G^{\star}$ on $\mathcal{S}(\bG^{\star},F^{\star})$ defined by $g\cdot (\bT^{\star},s) = ({}^g\bT^{\star},{}^gs)$, and we write $\mathcal{S}(\bG^{\star},F^{\star})/G^{\star}$ for the corresponding orbits and $[\bT^{\star},s]$ for the orbit containing $(\bT^{\star},s)$. By \cite[5.21.3]{deligne-lusztig:1976:representations-of-reductive-groups}, see also \cite[13.13]{digne-michel:1991:representations-of-finite-groups-of-lie-type}, we have a bijection
\begin{equation*}
\Pi : \mathcal{C}(\bG,F)/G \to \mathcal{S}(\bG^{\star},F^{\star})/G^{\star}
\end{equation*}
between these orbits. Note that this bijection depends on the choice of a group isomorphism $\imath : (\mathbb{Q}/\mathbb{Z})_{p'} \to \mathbb{K}^{\times}$ and an injective group homomorphism $\jmath : \mathbb{Q}/\mathbb{Z} \hookrightarrow \Ql^{\times}$, so we implicitly assume that such homomorphisms have been chosen.
\end{pa}

\begin{pa}
For any semisimple element $s \in G^{\star}$ we denote by $\mathcal{C}(\bG,F,s) \subseteq \mathcal{C}(\bG,F)$ the set of all pairs $(\bT,\theta)$ such that $\Pi([\bT,\theta]) = [\bT^\star,t]$ and $t$ is $G^{\star}$-conjugate to $s$. Now, to each pair $(\bT,\theta) \in \mathcal{C}(\bG,F)$, there is a corresponding Deligne--Lusztig character $R_{\bT}^{\bG}(\theta)$, and we denote by $\mathcal{E}(G,\bT,\theta)$ the set $\{\chi \in \Irr(G) \mid \langle \chi, R_{\bT}^{\bG}(\theta)\rangle_G \neq 0\}$ of its irreducible constituents. Note we will sometimes also write $R_{\bT^{\star}}^{\bG}(s)$ for $R_{\bT}^{\bG}(\theta)$ when $\Pi([\bT,\theta]) = [\bT^{\star},s]$. The union
\begin{equation*}
\mathcal{E}(G,s) = \bigcup_{(\bT,\theta) \in \mathcal{C}(\bG,F,s)} \mathcal{E}(G,\bT,\theta)
\end{equation*}
is, by definition, a rational Lusztig series. The set of all irreducible characters is then a disjoint union $\Irr(G) = \bigcup \mathcal{E}(G,s)$, where we run over all $G^{\star}$-conjugacy classes of semisimple elements, see \cite[11.8]{bonnafe:2006:sln}. If $H$ is a finite group and $x \in H$ is an element then we denote by $x_{\ell}$, resp., $x_{\ell'}$, the $\ell$-part, resp., $\ell'$-part, of $x = x_{\ell}x_{\ell'} = x_{\ell'}x_{\ell}$. With this we have the following.
\end{pa}

\begin{lemma}\label{lem:galoisLseries}
Let $s\in G^{\star}$ be a semisimple element and let $b, b' \in \mathbb{Z}$ be integers. If $\gamma\in\gal(\Q_{|G|}/\Q)$ is an automorphism such that $\gamma(\zeta) = \zeta^{\ell^b}$ for all $\ell'$-roots of unity and $\gamma(\zeta) = \zeta^{b'}$ for all $\ell$-roots of unity, then $\mathcal{E}(G, s)^\gamma = \mathcal{E}(G, s_\ell^{b'}s_{\ell'}^{\ell^b})$.
\end{lemma}

\begin{proof}
Assume $(\bT,\theta) \in \mathcal{C}(\bG,F)$.  Then by the character formula for $R_{\bT}^{\bG}(\theta)$ \cite[7.2.8]{carter:1993:finite-groups-of-lie-type}, and the fact that Green functions are integral valued, we easily deduce that $R_{\bT}^{\bG}(\theta)^{\gamma} = R_{\bT}^{\bG}(\theta^{\gamma})$. In particular, as $\gamma$ is an isometry we have $\mathcal{E}(\bG,\bT,\theta)^{\gamma} = \mathcal{E}(\bG,\bT,\theta^{\gamma})$. Now, if $\Pi([\bT,\theta]) = [\bT^{\star},s]$, then it is an easy consequence of the description of the map $\Pi$, see \cite[\S13]{digne-michel:1991:representations-of-finite-groups-of-lie-type}, and the definition of $\gamma$ that $\Pi([\bT,\theta^{\gamma}]) = [\bT^{\star}, s_\ell^{b'}s_{\ell'}^{\ell^b}]$. In particular this shows that $\mathcal{E}(G,s)^{\gamma} \subseteq \mathcal{E}(G,s_\ell^{b'}s_{\ell'}^{\ell^b})$. An almost identical argument shows that $\mathcal{E}(G,s) \subseteq \mathcal{E}(G,s_\ell^{b'}s_{\ell'}^{\ell^b})^{\gamma^{-1}} \subseteq \mathcal{E}(G,t)$ for some semisimple element $t \in G^{\star}$. However, by the disjointness of the rational series we must have equality which proves the lemma.
\end{proof}

\begin{pa}
For any irreducible character $\chi \in \Irr(G)$ we denote by $\omega_{\chi} : Z(G) \to \Ql^{\times}$ the central character determined by $\chi$. This is a linear character defined by $\omega_{\chi}(z) = \chi(z)/\chi(1)$ for any $z \in Z(G)$. The following will prove to be useful later; it follows from \cite[11.1(d)]{bonnafe:2006:sln}.
\end{pa}

\begin{lem}\label{lem:central-char}
For any two irreducible characters $\chi$, $\psi \in \mathcal{E}(G,s)$ we have $\omega_{\chi} = \omega_{\psi}$. In particular, if $\gamma\in\gal(\Q_{|G|}/\Q)$ is an automorphism and $\mathcal{E}(G,s)^{\gamma} = \mathcal{E}(G,s)$ then $\omega_{\chi}^{\gamma} = \omega_{\chi^{\gamma}} = \omega_{\chi}$ for all $\chi \in \mathcal{E}(G,s)$.
\end{lem}

\begin{pa}
Trying to understand the action of the Galois group on the elements of a rational Lusztig series is, in general, a difficult problem. However, in this section we will deal with two special cases. To describe these cases we need to introduce some notation. For $s \in G^{\star}$ a semisimple element, we denote by $\bT_s^{\star} \leqslant \bG^{\star}$ a fixed $F^{\star}$-stable maximal torus containing $s$; note that we then have $\bT_s^{\star}$ is contained in the centraliser $C_{\bG^{\star}}(s)$. We denote by $W^{\circ}(s) = N_{C_{\bG^{\star}}^{\circ}(s)}(\bT_s^{\star})/\bT_s^{\star}$ the Weyl group of the connected centraliser with respect to this maximal torus. For each $w \in W^{\circ}(s)$ we choose an $F^{\star}$-stable maximal torus $\bT_{s,w}^{\star} = {}^g\bT_s^{\star} \leqslant C_{\bG^{\star}}^{\circ}(s)$, where $g \in C_{\bG^{\star}}^{\circ}(s)$ is an element such that $g^{-1}F^{\star}(g) \in N_{C_{\bG^{\star}}^{\circ}(s)}(\bT_s^{\star})$ represents $w$. By \cite[15.11]{bonnafe:2006:sln} there then exists a sign such that
\begin{equation*}
\rho_s = \pm\frac{1}{|W^{\circ}(s)|}\sum_{w \in W^{\circ}(s)} R_{\bT_{s,w}^{\star}}^{\bG}(s),
\end{equation*}
is a character of $G$. Each irreducible constituent of this character is contained in the rational Lusztig series $\mathcal{E}(G,s)$ and is a semisimple character. Recall that a character is called semisimple if it is contained in the Alvis--Curtis dual of a Gelfand--Graev character, see \cite[8.8, 14.39]{digne-michel:1991:representations-of-finite-groups-of-lie-type}. We first consider the action of the Galois group on these characters.
\end{pa}

\begin{prop}\label{prop:semisimple-chars-sigma-fixed}
Let $\gamma$ be as in \cref{lem:galoisLseries} and assume $s \in G^{\star}$ is a semisimple element such that $\mathcal{E}(G,s)^{\gamma} = \mathcal{E}(G,s)$, then the following hold:
\begin{enumerate}
	\item $\rho_s$ is fixed by $\gamma$,
	\item every semisimple character contained in $\mathcal{E}(G,s)$ is fixed by $\gamma$ if every Gelfand--Graev character of $G$ is fixed by $\gamma$.
\end{enumerate}
\end{prop}
\begin{proof}
If $\mathcal{E}(G,s)^{\gamma} = \mathcal{E}(G,s)$, then we have $s$ is $G^{\star}$-conjugate to $s_\ell^{b'}s_{\ell'}^{\ell^b}$. From the arguments in the proof of \cref{lem:galoisLseries} it is clear that, under this assumption, we have $R_{\bT_{s,w}^{\star}}^{\bG}(s)^{\gamma} = R_{\bT_{s,w}^{\star}}^{\bG}(s)$ so clearly $\rho_s$ is fixed by $\gamma$. Now, if $\Gamma$ is a Gelfand--Graev character of $G$ and $D_G$ denotes Alvis--Curtis duality, see \cite[8.8]{digne-michel:1991:representations-of-finite-groups-of-lie-type}, then there exists a unique irreducible constituent $\chi$ of $\rho_s$ such that $\langle D_G(\Gamma), \chi\rangle_G \neq 0$, see \cite[15.11]{bonnafe:2006:sln}. Certainly we have $\chi^{\gamma}$ is both a constituent of $\rho_s^{\gamma} = \rho_s$ and $D_G(\Gamma)^{\gamma}$. From the definition of $D_G$, and the character formula for Harish-Chandra induction/restriction \cite[4.5]{digne-michel:1991:representations-of-finite-groups-of-lie-type}, it is not difficult to see that $D_G(\Gamma)^{\gamma} = D_G(\Gamma^{\gamma})$. Hence, if $\Gamma^{\gamma} = \Gamma$ then we must have $\chi^{\gamma}$ is a constituent of $D_G(\Gamma)$; but this implies $\chi^{\gamma} = \chi$ by the uniqueness.
\end{proof}

\begin{pa}\label{pa:almost-chars}
The next case we wish to consider is that of $\GL_n(\mathbb{K})$.  First, we introduce some notation that will be useful later. Specifically, let $s \in G^{\star}$ be a semisimple element. Then the Frobenius $F^{\star}$ induces an automorphism $F^{\star} : W^{\circ}(s) \to W^{\circ}(s)$ because $\bT_s^{\star}$ is assumed to be $F^{\star}$-stable. We denote by $\widetilde{W}^{\circ}(s)$ the semidirect product $W^{\circ}(s) \rtimes\langle F^{\star} \rangle$ and for any class function $f : \widetilde{W}^{\circ}(s) \to \Ql$ we define a corresponding class function
\begin{equation*}
R^G_f(s) = \frac{1}{|W^{\circ}(s)|}\sum_{w \in W^{\circ}(s)} f(wF^{\star})R_{\bT_{s,w}^{\star}}^{\bG}(s)
\end{equation*}
of $G$. With this we can prove the following.
\end{pa}

\begin{prop}\label{prop:GLn-sigma-fixed}
Assume $\bG$ is $\GL_n(\mathbb{K})$, $\gamma$ is as in \cref{lem:galoisLseries}, and $s \in G^{\star}$ is a semisimple element such that $\mathcal{E}(G,s)^{\gamma} = \mathcal{E}(G,s)$.  Then every $\chi \in \mathcal{E}(G,s)$ is fixed by $\gamma$.
\end{prop}

\begin{proof}
By \cite[3.2, 4.23]{lusztig:1984:characters-of-reductive-groups} every irreducible character in the Lusztig series $\mathcal{E}(G,s)$ is of the form $R^G_f(s)$ where $f : \widetilde{W}^{\circ}(s) \to \Ql$ is a rational valued irreducible character, see also \cite[13.25(ii), \S15.4]{digne-michel:1991:representations-of-finite-groups-of-lie-type}. The statement now follows immediately from the fact that each $R_{\bT_{s,w}^{\star}}^{\bG}(s)$ is fixed by $\gamma$, c.f., the proof of \cref{prop:semisimple-chars-sigma-fixed}.
\end{proof}

\section{GGGRs and Galois Automorphisms}\label{sec:GGGRs}
\begin{assumption}
In this section, and in this section only, we assume that $p$ is a good prime for $\bG$ and that $\bG$ is a proximate algebraic group in the sense of \cite[2.10]{taylor:2016:GGGRs-small-characteristics}. Recall that this means some (any) simply connected covering of the derived subgroup of $\bG$ is seperable.
\end{assumption}

\begin{pa}\label{pa:GGGR-rat-conj}
To any unipotent element $u \in G$ Kawanaka has defined a corresponding generalised Gelfand--Graev representation (GGGR) of $G$ which we denote $\Gamma_u$, see \cite{kawanaka:1985:GGGRs-and-ennola-duality,taylor:2016:GGGRs-small-characteristics}. If $u$ is a regular unipotent element then $\Gamma_u$ is a Gelfand--Graev character. Moreover, we have $\Gamma_{gug^{-1}} = \Gamma_u$ for any $g \in G$. In this section we wish to determine the effect of $\sigma$ on the GGGRs of $G$; for this we must recall their construction. Let $\lie{g}$ denote the Lie algebra of $\bG$ and let $\mathcal{N} \subseteq \mathfrak{g}$, resp., $\mathcal{U} \subseteq \bG$, denote the nilpotent cone of $\mathfrak{g}$, resp., the unipotent variety of $\bG$. The Frobenius endomorphism $F : \bG \to \bG$ induces a corresponding Frobenius endomorphism $F : \lie{g} \to \lie{g}$ on the Lie algebra. We have $F(\mathcal{U}) = \mathcal{U}$ and $F(\mathcal{N}) = \mathcal{N}$.
\end{pa}

\begin{pa}
Let $\mathbb{G}_m$ denote the set $\mathbb{K} \setminus \{0\}$ viewed as a multiplicative algebraic group and let $\widecheck{X}(\bG) = \Hom(\mathbb{G}_m,\bG)$ be the set of all cocharacters of $\bG$. Let $F_q : \mathbb{G}_m \to \mathbb{G}_m$ denote the Frobenius endomorphism given by $F_q(k) = k^q$, with $q$ as in \cref{pa:basic-setup}. Then for any $\lambda \in \widecheck{X}(\bG)$ we define a new cocharacter $F\cdot\lambda \in \widecheck{X}(\bG)$ by setting
\begin{equation*}
(F\cdot\lambda)(k) = F(\lambda(F_q^{-1}(k)))
\end{equation*}
for all $k \in \mathbb{G}_m$. We denote by $\widecheck{X}(\bG)^F \subseteq \widecheck{X}(\bG)$ the set of all cocharacters $\lambda$ satisfying $F\cdot\lambda = \lambda$.
\end{pa}

\begin{pa}
To each cocharacter $\lambda \in \widecheck{X}(\bG)$ we have a corresponding parabolic subgroup $\bP(\lambda) \leqslant \bG$ with unipotent radical $\bU(\lambda) \leqslant \bP(\lambda)$ and Levi complement $\bL(\lambda) = C_{\bG}(\lambda(\mathbb{G}_m))$, see \cite[3.2.15, 8.4.5]{springer:2009:linear-algebraic-groups}. The group $\bG$ acts on $\lie{g}$ via the adjoint representation $\Ad : \bG \to \GL(\lie{g})$. Through $\Ad$ we have each cocharacter $\lambda$ defines a $\mathbb{Z}$-grading $\lie{g} = \bigoplus_{i \in \mathbb{Z}} \lie{g}(\lambda,i)$ on the Lie algebra. For any $i > 0$ we have $\lie{u}(\lambda,i) = \bigoplus_{j \geqslant i} \lie{g}(\lambda,j)$ is a subalgebra of the Lie algebra of $\bU(\lambda)$ and it is the Lie algebra of a closed connected subgroup $\bU(\lambda,i) \leqslant \bU(\lambda)$ which is normal in $\bP(\lambda)$. The group $\bL(\lambda)$ preserves each weight space $\lie{g}(\lambda,i)$ and we denote by $\lie{g}(\lambda,2)_{\mathrm{reg}} \subseteq \lie{g}(\lambda,2)$ the unique open dense orbit of $\bL(\lambda)$ acting on $\lie{g}(\lambda,2)$. Note that if $\lambda \in \widecheck{X}(\bG)^F$ then the subgroups $\bP(\lambda)$, $\bU(\lambda)$, $\bU(\lambda,i)$, and $\bL(\lambda)$ are all $F$-stable and we set $P(\lambda) = \bP(\lambda)^F$, $U(\lambda) = \bU(\lambda)^F$, $U(\lambda,i) = \bU(\lambda,i)^F$, and $L(\lambda) = \bL(\lambda)^F$.
\end{pa}

\begin{pa}
The action of $\bG$ on $\lie{g}$ preserves $\mathcal{N}$ and the action of $\bG$ on itself by conjugation preserves $\mathcal{U}$; we denote the resulting sets of orbits by $\mathcal{N}/\bG$ and $\mathcal{O}/\bG$. Recall that each nilpotent orbit $\mathcal{O} \in \mathcal{N}/\bG$ is of the form $\mathcal{O} = (\Ad\bG)\lie{g}(\lambda,2)_{\mathrm{reg}}$ for some $\lambda \in \widecheck{X}(\bG)$, see \cite[3.22]{taylor:2016:GGGRs-small-characteristics}. Moreover, if $\mathcal{O}$ is $F$-stable then we may assume that $\lambda \in \widecheck{X}(\bG)^F$, see \cite[3.25]{taylor:2016:GGGRs-small-characteristics}. Following \cite[\S4, \S5]{taylor:2016:GGGRs-small-characteristics} we assume a chosen $\bG$-equivariant isomorphism of varieties $\phi_{\mathrm{spr}} : \mathcal{U} \to \mathcal{N}$ which commutes with $F$ and whose restriction to each $\bU(\lambda)$ is a Kawanaka isomorphism. In particular, the map $\phi_{\mathrm{spr}}$ satisfies the following two properties:
\begin{enumerate}[label=(K\arabic*)]
	\item $\phi_{\mathrm{spr}}(\bU(\lambda,2)) \subseteq \lie{u}(\lambda,2)$,
	\item $\phi_{\mathrm{spr}}(uv) - \phi_{\mathrm{spr}}(u) - \phi_{\mathrm{spr}}(v) \in \mathfrak{u}(\lambda,3)$ for any $u,v \in \bU(\lambda,2)$.
\end{enumerate}
Note also that $\phi_{\mathrm{spr}}$ induces a bijection $\mathcal{U}/\bG \to \mathcal{N}/\bG$. Before introducing the GGGRs we consider the following lemmas, which were not covered in \cite{taylor:2016:GGGRs-small-characteristics}.
\end{pa}

\begin{lem}\label{lem:kaw-iso-eq}
For each cocharacter $\lambda \in \widecheck{X}(\bG)$ we have $\phi_{\mathrm{spr}}(\bU(\lambda,2)) = \mathfrak{u}(\lambda,2)$.
\end{lem}

\begin{proof}
As $\phi_{\mathrm{spr}}$ is an isomorphism we have $\phi_{\mathrm{spr}}(\bU(\lambda,2))$ is a closed subset of the same dimension as $\mathfrak{u}(\lambda,2)$. As $\mathfrak{u}(\lambda,2)$ is irreducible we must have $\phi_{\mathrm{spr}}(\bU(\lambda,2)) = \mathfrak{u}(\lambda,2)$.
\end{proof}

\begin{lem}\label{lem:P-conj}
Assume $\mathcal{O} \in \mathcal{U}/\bG$ is such that $\phi_{\mathrm{spr}}(\mathcal{O}) = (\Ad\bG)\lie{g}(\lambda,2)_{\mathrm{reg}}$ for some cocharacter $\lambda \in \widecheck{X}(\bG)$. Then $\mathcal{O} \cap \bU(\lambda,2)$ is an open dense subset of $\bU(\lambda,2)$ and is a single $\bP(\lambda)$-conjugacy class.
\end{lem}

\begin{proof}
Choose an element $e \in \lie{g}(\lambda,2)_{\mathrm{reg}}$ and let $u \in \mathcal{U}$ be the unique unipotent element satisfying $\phi_{\mathrm{spr}}(u) = e$. By \cref{lem:kaw-iso-eq} we have $u \in \bU(\lambda,2)$ so the $\bP(\lambda)$-conjugacy class $\mathcal{O}_{\bP(\lambda)}$ containing $u$ is contained in $\mathcal{O} \cap\bU(\lambda,2) \subseteq \bU(\lambda,2)$. We thus clearly have a corresponding sequence of closed sets
\begin{equation*}
\overline{\mathcal{O}_{\bP(\lambda)}} \subseteq \overline{\mathcal{O}\cap \bU(\lambda,2)} \subseteq \overline{\mathcal{O}}\cap \bU(\lambda,2) \subseteq \bU(\lambda,2).
\end{equation*}
According to \cite[3.22(ii.b)]{taylor:2016:GGGRs-small-characteristics} we have $\overline{\phi_{\mathrm{spr}}(\mathcal{O}_{\bP(\lambda)})} = \overline{(\Ad\bP(\lambda))e} = \mathfrak{u}(\lambda,2)$. As $\phi_{\mathrm{spr}}$ is an isomorphism it follows from \cref{lem:kaw-iso-eq} that $\overline{\mathcal{O}_{\bP(\lambda)}} = \bU(\lambda,2)$ so all of these containments above must be equalities. This certainly shows $\mathcal{O} \cap \bU(\lambda,2)$ is dense and as $\mathcal{O}$ is open in $\overline{\mathcal{O}}$ we have the intersection is also open.

Let $v \in \mathcal{O} \cap\bU(\lambda,2)$ be another element in the intersection and denote by $\mathcal{O}' \subseteq \mathcal{O} \cap\bU(\lambda,2)$ the $\bP(\lambda)$-conjugacy class containing $v$. As $v$ is $\bG$-conjugate to $u$ we have $\dim C_{\bG}(v) = \dim C_{\bG}(u)$ so
\begin{equation*}
\dim\mathcal{O}' = \dim\bP(\lambda) - \dim C_{\bP(\lambda)}(v) \geqslant \dim\bP(\lambda) - \dim C_{\bG}(u) = \dim\bU(\lambda,2),
\end{equation*}
where the last equality follows from \cite[3.22(ii)]{taylor:2016:GGGRs-small-characteristics}. As $\mathcal{O}' \subseteq \bU(\lambda,2)$ we must have $\dim \mathcal{O}' = \dim\bU(\lambda,2)$ so $\overline{\mathcal{O}'} = \bU(\lambda,2)$, because $\bU(\lambda,2)$ is irreducible, and $\mathcal{O}'$ is also a dense open subset of $\bU(\lambda,2)$. Again, as $\bU(\lambda,2)$ is irreducible this implies $\mathcal{O}_{\bP(\lambda)} \cap \mathcal{O}' \neq \emptyset$ which shows $\mathcal{O}_{\bP(\lambda)} = \mathcal{O} \cap \bU(\lambda,2)$.
\end{proof}

\begin{cor}\label{cor:rat-P-conj}
Let $u \in \mathcal{U}^F$ be a rational unipotent element and let $\mathcal{O} \in \mathcal{U}/\bG$ be the $F$-stable class containing $u$. If $\lambda \in \widecheck{X}(\bG)^F$ is such that $\phi_{\mathrm{spr}}(\mathcal{O}) = (\Ad\bG)\lie{g}(\lambda,2)_{\mathrm{reg}}$ then any element contained in $\mathcal{O}\cap U(\lambda,2)$ is of the form ${}^{hl}u$ with $h \in U(\lambda)$ and $l \in \bL(\lambda)$.
\end{cor}

\begin{proof}
Assume $v \in \mathcal{O}\cap U(\lambda,2)$, so by \cref{lem:P-conj} there exists an element $g \in \bP(\lambda)$ such that $v = {}^gu$. As $F(v) = v$ we must have $g^{-1}F(g) \in C_{\bP(\lambda)}(u)$. If we set $A_{\bP(\lambda)}(u) = C_{\bP(\lambda)}(u)/C_{\bP(\lambda)}^{\circ}(u)$ then the map ${}^gu \mapsto g^{-1}F(g)C_{\bP(\lambda)}^{\circ}(u)$ induces a bijection between the $P(\lambda)$-conjugacy classes contained in $\mathcal{O}\cap U(\lambda,2) = (\mathcal{O} \cap \bU(\lambda,2))^F$ and the $F$-conjugacy classes of $A_{\bP(\lambda)}(u)$, see \cite[4.3.5]{geck:2003:intro-to-algebraic-geometry}. If $A_{\bL(\lambda)}(u) = C_{\bL(\lambda)}(u)/C_{\bL(\lambda)}^{\circ}(u)$ then it's known that the embedding $C_{\bL(\lambda)}(u) \hookrightarrow C_{\bP(\lambda)}(u)$ induces an isomorphism $A_{\bL(\lambda)}(u) \to A_{\bP(\lambda)}(u)$. Indeed, arguing as in the proof of \cite[3.22]{taylor:2016:GGGRs-small-characteristics} we obtain from \cite[2.3]{premet:2003:nilpotent-orbits-in-good-characteristic} that $C_{\bP(\lambda)}(u) = C_{\bL(\lambda)}(u) \ltimes C_{\bU(\lambda)}(u)$ from which the statement follows immediately. Applying the Lang--Steinberg theorem to the connected group $\bL(\lambda)$ there exists an element $l_1 \in \bL(\lambda)$ such that $l_1^{-1}F(l_1)C_{\bP(\lambda)}^{\circ}(u) = g^{-1}F(g)C_{\bP(\lambda)}^{\circ}(u)$. We therefore have ${}^{l_1}u$ and $v$ are $P(\lambda)$ conjugate. As $P(\lambda) = U(\lambda) \rtimes L(\lambda)$ the statement follows.
\end{proof}

\begin{pa}\label{pa:GGGR-intro}
We are now ready to introduce GGGRs. For this we assume a chosen $\bG$-invariant trace form $\kappa(-,-) : \lie{g} \times \lie{g} \to \mathbb{K}$, which is not too degenerate in the sense of \cite[5.6]{taylor:2016:GGGRs-small-characteristics}, and an $\mathbb{F}_q$-opposition automorphism ${}^{\dag} : \lie{g} \to \lie{g}$, see \cite[5.1]{taylor:2016:GGGRs-small-characteristics} for the definition. Moreover, we assume $\chi_q : \mathbb{F}_q^+ \to \Ql^{\times}$ is a character of the finite field $\mathbb{F}_q$ viewed as an additive group. Let $u \in \mathcal{U}^F$ be a rational unipotent element and let $\lambda \in \widecheck{X}(\bG)^F$ be a cocharacter such that $e = \phi_{\mathrm{spr}}(u) \in \lie{g}(\lambda,2)_{\mathrm{reg}}$. Following \cite[5.10]{taylor:2016:GGGRs-small-characteristics} we define a linear character $\varphi_u : U(\lambda,2) \to \Ql$ by setting
\begin{equation*}
\varphi_u(x) = \chi_q(\kappa(e^{\dag},\phi_{\mathrm{spr}}(x))).
\end{equation*}
With this we have the following definition of the GGGR $\Gamma_u$.
\end{pa}

\begin{definition}
The index $[U(\lambda,1):U(\lambda,2)]$ is an even power of $q$ and the class function
\begin{equation*}
\Gamma_u = [U(\lambda,1):U(\lambda,2)]^{-1/2}\Ind_{U(\lambda,2)}^G(\varphi_u).
\end{equation*}
is a character of $G$ known as a generalised Gelfand--Graev representation (GGGR).
\end{definition}

\begin{prop}\label{prop:galois-aut-GGGR}
Let $\gamma \in \Gal(\mathbb{Q}_{|G|}/\mathbb{Q})$ be a Galois automorphism such that $\gamma(\zeta) = \zeta^n$ for all $p$-roots of unity, where $n \in \mathbb{Z}$ is an integer coprime to $p$. Then for any unipotent element $u \in \mathcal{U}^F$ we have $\Gamma_u^{\gamma} = \Gamma_{u^n}$.
\end{prop}

\begin{proof}
We assume $e$ and $\lambda$ are as in \cref{pa:GGGR-intro}. Let $\mathcal{O} \in \mathcal{U}/\bG$ be the class containing $u$. As $n$ is coprime to $p$ we have $u$ and $u^n$ generate the same cyclic subgroup of $\bG$ so $u^n \in \mathcal{O}$ by \cite[Corollary 3]{liebeck-seitz:2012:unipotent-nilpotent-classes}. Now clearly $u^n \in \mathcal{O} \cap U(\lambda,2)$ so by \cref{cor:rat-P-conj} there exist elements $h \in U(\lambda)$ and $l \in \bL(\lambda)$ such that $u^n = {}^{hl}u$. We thus have $\phi_{\mathrm{spr}}(u^n) = \phi_{\mathrm{spr}}({}^{hl}u) = (\Ad hl)e$.

By property (K2) above we have $\phi_{\mathrm{spr}}(u^n) \equiv ne \pmod{\mathfrak{u}(\lambda,3)}$. As $\phi_{\mathrm{spr}}(u^n) = (\Ad hl)e$ and $h \in U(\lambda)$ we conclude from \cite[Lemma 10]{mcninch:2004:nilpotent-orbits-good-characteristic} that
\begin{equation*}
(\Ad l)e \equiv ne \pmod{\mathfrak{u}(\lambda,3)}.
\end{equation*}
However, as $\bL(\lambda)$ preserves each weight space we have $(\Ad l)e \in \lie{g}(\lambda,2)$ so it must be that $(\Ad l)e = ne$. As mentioned in \cref{pa:GGGR-rat-conj} we have $\Gamma_{u^n} = \Gamma_{{}^{hl}u} = \Gamma_{{}^lu}$ so it is sufficient to show that $\Gamma_u^{\gamma} = \Gamma_{{}^lu}$. Clearly $\phi_{\mathrm{spr}}({}^lu) = (\Ad l)e \in \lie{g}(\lambda,2)_{\mathrm{reg}}$ so it is sufficient from the definition of the GGGR to show that $\varphi_u^{\gamma} = \varphi_{{}^lu}$.

As $\mathbb{F}_q^+$ is an abelian $p$-group and $\chi_q : \mathbb{F}_q^+ \to \Ql$ is a homomorphism it is clear that $\chi_q(a)^{\gamma} = \chi_q(na)$ for any $a \in \mathbb{F}_q^+$. Now, for any $x \in U(\lambda,2)$ we thus have
\begin{equation*}
\varphi_u^{\gamma}(x) = \chi_q(n\kappa(e^{\dag},\phi_{\mathrm{spr}}(x))) = \chi_q(\kappa((ne)^{\dag},\phi_{\mathrm{spr}}(x))) = \varphi_{{}^lu}(x)
\end{equation*}
as desired.
\end{proof}

\section{\cref{cond:conjIF} when $p=2$}\label{sec:p2IF}
\begin{assumption}
In this section and the following section we assume that $p=2$.
\end{assumption}

\begin{pa}
In \cite[\S4.1]{schaeffer-fry:2015:odd-degree-characters} the first author showed that $G$ satisfies \cref{cond:conjIF} in most cases where $G$ is a quasisimple group. The purpose of this section is to complete this work to show that all quasisimple groups of Lie type in characteristic 2 satisfy \cref{cond:conjIF}. We will do this using a general statement which describes precisely which odd degree characters of $G$ are fixed by $\sigma$. Note the techniques and ideas we use here are a synthesis of those already used in \cite{schaeffer-fry:2015:odd-degree-characters}. When $q > 2$ these characters are always semisimple and we may apply \cref{prop:semisimple-chars-sigma-fixed}, which generalises \cite[4.6]{schaeffer-fry:2015:odd-degree-characters}. When $q = 2$ not all odd degree characters are semisimple and we must provide some additional ad-hoc arguments to deal with these cases.
\end{pa}

\begin{lem}[Malle, {\cite[6.8]{malle:2007:height-0-characters}}]\label{lem:malle}
Assume either that $q>2$ or the Dynkin diagram of $\bG$ is simply laced then the only odd degree unipotent character is the trivial character.
\end{lem}

\begin{prop}\label{prop:odd-degree-fixed}
An odd degree character $\chi \in \mathcal{E}(G,s)$ is $\sigma$-fixed if and only if $s$ is $G^{\star}$-conjugate to $s^2$.
\end{prop}

\begin{proof}
Let $\chi \in \mathcal{E}(G,s)$ be an irreducible character of $G$ of odd degree and choose an irreducible character $\widetilde{\chi} \in \Irr(\widetilde{G} |\chi)$ covering $\chi$. By \cite[Proposition 10]{lusztig:1988:reductive-groups-with-a-disconnected-centre} the restriction $\Res_G^{\widetilde{G}}(\widetilde{\chi})$ is multiplicity free so $\widetilde{\chi}(1) = [\widetilde{G}:I_{\widetilde{G}}(\chi)]\chi(1)$, where $G \leqslant I_{\widetilde{G}}(\chi)$ is the inertia group of $\chi$. The order of the quotient $\widetilde{G}/G$ is coprime to $p$, hence so is $[\widetilde{G}:I_{\widetilde{G}}(\chi)]$. This implies $\widetilde{\chi}(1)$ is odd.

Now, assume $\widetilde{\chi}$ is contained in the Lusztig series $\mathcal{E}(\widetilde{G},\widetilde{s})$ then by \cite[4.23]{lusztig:1984:characters-of-reductive-groups} there exists a bijection $\Psi_{\widetilde{s}} : \mathcal{E}(\widetilde{G},\widetilde{s}) \to \mathcal{E}(C_{\widetilde{G}^{\star}}(\widetilde{s}),1)$ such that $\widetilde{\chi}(1) = [\widetilde{G} : C_{\widetilde{G}^{\star}}(\widetilde{s})]_{p'}\Psi_{\widetilde{s}}(\widetilde{\chi})(1)$, see also \cite[13.23, 13.24]{digne-michel:1991:representations-of-finite-groups-of-lie-type}. As $\widetilde{\chi}(1)$ is odd we must therefore have that $\Psi_s(\widetilde{\chi})(1)$ is also odd.

Let us assume, for the moment, that $q > 2$. Then according to \cref{lem:malle}, there is only one unipotent character of $C_{\widetilde{G}^{\star}}(\widetilde{s})$ of odd degree, namely the trivial character. Consequently, this implies that $\mathcal{E}(\widetilde{G},\widetilde{s})$ contains a unique character of odd degree and so $\widetilde{\chi}$ must be the unique semisimple character contained in this series, see \cite[8.4.8]{carter:1993:finite-groups-of-lie-type}. The character $\chi$ must therefore also be semisimple. Now any Gelfand--Graev character of $G$ is obtained by inducing a linear character from a Sylow $p$-subgroup of $G$. As $p=2$ this implies all Gelfand--Graev characters are $\sigma$-fixed so $\chi^{\sigma} = \chi$ by \cref{prop:semisimple-chars-sigma-fixed}.

We now assume that $q = 2$. If the Dynkin diagram of $C_{\widetilde{\bG}^{\star}}(\widetilde{s})$ is simply laced then we may apply the previous argument; so assume this is not the case. The Dynkin diagram of $\bG$ must then also have a component which is not simply laced. This corresponds to a semisimple subgroup of $\bG$ which has a trivial centre so splits off as a direct factor. With this it is clear that we need only consider the case where $\bG$ is simple of type $\B_n$, $\C_n$, $\F_4$, or $\G_2$.

Let $\mathcal{F} \subseteq \Irr(W^{\circ}(s))$ be an $F^{\star}$-stable family of characters of the Weyl group of $C_{\bG^{\star}}(s) = C_{\bG^{\star}}^{\circ}(s)$. For each $F^{\star}$-fixed character in $\mathcal{F}$ we choose one of its extensions to $\widetilde{W}^{\circ}(s)$ which is defined over $\mathbb{Q}$, c.f., \cite[3.2]{lusztig:1984:characters-of-reductive-groups}, and denote by $\widetilde{\mathcal{F}} \subseteq \Irr(\widetilde{W}^{\circ}(s))$ the resulting set of extensions. According to \cite[4.23]{lusztig:1984:characters-of-reductive-groups} there is a unique family such that $\langle \chi, R_f^G(s)\rangle_G \neq 0$ for some $f \in \widetilde{\mathcal{F}}$, c.f., \cref{pa:almost-chars}. Now as each $f \in \widetilde{\mathcal{F}}$ is rational valued we see that
\begin{equation*}
\langle \chi, R_f^G(s) \rangle_G = \langle \chi^{\sigma}, R_f^G(s)^{\sigma} \rangle_G = \langle \chi^{\sigma}, R_f^G(s) \rangle_G
\end{equation*}
If $\bG$ is of type $\B_n$ or $\C_n$ then these multiplicities uniquely determine the character $\chi$ so we must have $\chi = \chi^{\sigma}$ in these cases, see \cite[6.3]{digne-michel:1990:lusztigs-parametrization}. This statement is not true in general when $\bG$ is of type $\G_2$ or $\F_4$. However, comparing the tables of unipotent characters in \cite[\S13.9]{carter:1993:finite-groups-of-lie-type} with \cite[6.3]{digne-michel:1990:lusztigs-parametrization} we see the statement still holds for those of odd degree.
\end{proof}

\begin{assumption}
From now until the end of this article we assume that $\bG$ is simple and simply connected.
\end{assumption}

\begin{pa}\label{pa:assump-sylow}
If $G$ is perfect then the quotient $S = G/Z$ is a simple group of Lie type defined in characteristic $2$. We now wish to show that $G$ satisfies \cref{cond:conjIF}. With regards to this let $Q \leqslant \Aut(G)$ be a $2$-group which stabilises a Sylow $P \in \syl_2(G)$. The normaliser $B_0 = N_G(P)$ is a Borel subgroup of $G$, c.f., \cite[2.29(i)]{cabanes-enguehard:2004:representation-theory-of-finite-reductive-groups}, because $p=2$. We may clearly replace $P$ and $Q$ by any $G$-conjugate so we may assume that $B_0$ contains our fixed maximal torus $T_0$, c.f., \cref{pa:basic-setup}. In particular, we have $B_0 = P \rtimes T_0$ so $N_G(P)/P \cong T_0$. Note that as $Q$ stabilises $P$ it also stabilises $B_0$ and hence also $T_0$. We will denote by $\bB_0 \leqslant \bG$ an $F$-stable Borel subgroup such that $B_0 = \bB_0^F$.
\end{pa}

\begin{pa}
As we are working in characteristic $2$, we have to be careful when dealing with small fields. Namely we have to be mindful of degenerate tori, in the sense of \cite[3.6.1]{carter:1993:finite-groups-of-lie-type}. For instance, it can happen when $q=2$ that the torus $T_0$ is the trivial subgroup, c.f., \cite[3.6.7]{carter:1993:finite-groups-of-lie-type}. The following shows that $\bT_0$ is degenerate only when $q = 2$.
\end{pa}

\begin{lem}\label{lem:non-degen-tori}
The maximal torus $\bT_0$ is non-degenerate if and only if $q > 2$ or $G$ is ${}^2\A_n(2)$ with $n \geqslant 2$.
\end{lem}

\begin{proof}
To show that $\bT_0$ is non-degenerate we must show that for any root $\alpha \in \Phi \subseteq X(\bT_0)$ there exists an element $t \in T_0$ such that $\alpha(t) \neq 1$.

We start by treating the case where $G$ is of type ${}^2\A_n(q)$ with $n \geqslant 2$. We may assume that $\bG = \SL_{n+1}(\mathbb{K})$ and $\bT_0 \leqslant \bB_0$ are the subgroups of diagonal matrices and upper triangular matrices respectively. Moreover, we assume that $F = F_q \circ\phi = \phi \circ F_q$ where $F_q : \bG \to \bG$ is the Frobenius endomorphism raising each matrix entry to the power $q$ and $\phi : \bG \to \bG$ is the automorphism defined by $\phi(x) = (x^{-T})^{n_0}$, where $n_0 \in N_{\bG}(\bT_0)$ is the permutation matrix representing the longest element in the symmetric group. For any $1 \leqslant i \leqslant n$ we consider the usual homomorphisms $\varepsilon_i : \bT_0 \to \mathbb{K}^{\times}$ and $\widecheck{\varepsilon}_i : \mathbb{K}^{\times} \to \bT_0$ such that $\{\pm\varepsilon_i \mp \varepsilon_j \mid 1 \leqslant i < j \leqslant n+1\}$ is the set of roots and $\{\pm\widecheck{\varepsilon}_i \mp \widecheck{\varepsilon}_j \mid 1 \leqslant i < j \leqslant n+1\}$ is the set of coroots.

Given an element $\zeta \in \mathbb{F}_{q^2}^{\times} \leqslant \mathbb{K}^{\times}$ and an integer $1 \leqslant i \leqslant n+1$ we define a corresponding element
\begin{equation*}
t_i(\zeta) = (\widecheck{\varepsilon}_i - q\widecheck{\varepsilon}_{n+2-i})(\zeta) \in T_0.
\end{equation*}
Now assume $\alpha = \varepsilon_i - \varepsilon_j$ with $1 \leqslant i < j \leqslant n+1$. If $j = n+2-i$ then we have $\alpha(t_i(\zeta)) = \zeta^2$ and if $j \neq n+2-i$ then we have $\alpha(t_i(\zeta^{-1})t_{n+2-j}(\zeta)) = \zeta^{q-1}$. Thus, as we can clearly choose $\zeta \not\in \mathbb{F}_q^{\times}$ we see that $\bT_0$ is always non-degenerate. With this case dealt with we may assume that $G$ is not of type ${}^2\A_n(q)$ with $n \geqslant 2$.

Now let us denote by $\langle-,-\rangle : X(\bT_0) \times \widecheck{X}(\bT_0)$ the usual perfect pairing between the character and cocharacter groups of $\bT_0$. Let $\tau : \Phi \to \Phi$ and $\widecheck{\tau} : \widecheck{\Phi} \to \widecheck{\Phi}$ be the permutation of the roots and coroots induced by $F$. Given $\alpha \in \Phi$ we denote by $k\geqslant 1$ the smallest integer such that $\widecheck{\tau}^k(\widecheck{\alpha}) = \widecheck{\alpha}$. Given an element $\zeta \in \mathbb{F}_{q^k}^{\times} \leqslant \mathbb{K}^{\times}$ we define a corresponding element $t_{\alpha}(\zeta) \in T$ by setting
\begin{equation*}
t_{\alpha}(\zeta) = \widecheck{\alpha}(\zeta)\cdot\widecheck{\tau}(\widecheck{\alpha})(\zeta^q) \cdots \widecheck{\tau}^{k-1}(\widecheck{\alpha})(\zeta^{q^{k-1}})
\end{equation*}
As we assume that $G$ is not of type ${}^2\A_n(q)$ we have by \cite[10.3.2(iii)]{springer:2009:linear-algebraic-groups} that $\langle\alpha,\widecheck{\tau}^i(\widecheck{\alpha})\rangle = 0$ for any $1 \leqslant i \leqslant k-1$ and so
\begin{equation*}
\alpha(t_{\alpha}(\zeta)) = \zeta^{\langle \alpha,\widecheck{\alpha}\rangle}\zeta^{q\langle \alpha,\widecheck{\tau}(\widecheck{\alpha})\rangle}\cdots \zeta^{q^{k-1}\langle \alpha,\widecheck{\tau}^{k-1}(\widecheck{\alpha})\rangle} = \zeta^2.
\end{equation*}
Hence, if $\mathbb{F}_{q^k}^{\times}$ contains a non-trivial element then we have the torus is non-degenerate. This is the case if $q>2$.

Now assume that $q=2$. If $F$ is split then we have $T_0 = \{1\}$ by \cite[3.6.7]{carter:1993:finite-groups-of-lie-type}, so certainly the torus is degenerate in this case. Finally, it is an easy exercise with root systems to show that $\bT_0$ is degenerate when $G$ is ${}^2\D_n(2)$ ($n \geqslant 4$), ${}^3\D_4(2)$, or ${}^2\E_6(2)$. We leave the details to the reader.
\end{proof}

\begin{pa}
As $\bG$ is simply connected, any automorphism of $G$ can be obtained by restricting a bijective morphism of $\bG$ which commutes with $F$. Now recall that, with respect to $\bT_0$ and $\bB_0$, we have the notions of a graph, field, and diagonal automorphism, see \cite[Theorem 30, pg.\ 158]{steinberg:1968:lectures-on-chevalley-groups}. In particular, the automorphism $x \mapsto x^2$ of $\mathbb{K}$ determines a bijective morphism of $\bG$ that generates the cyclic subgroup of all field automorphisms. We refer to this automorphism as a \emph{generating field automorphism}. Now any $\varphi \in \Aut(G)$ can be written as a product $\alpha\beta\gamma\delta$ where $\alpha$ is an inner automorphism, $\beta$ is a field automorphism, $\gamma$ is a graph automorphism, and $\delta$ is a diagonal automorphism. We note, however, that graph automorphisms are omitted when $F$ is twisted, see \cite[Theorem 36, pg.\ 195]{steinberg:1968:lectures-on-chevalley-groups}. With these notions in place we have the following relating to \cref{cond:conjIF}.
\end{pa}

\begin{lem}\label{lem:Q-centraliser}
Keep the notation and assumptions of \cref{pa:assump-sylow} and furthermore assume that $\bT_0$ is nondegenerate. Then we have $C_{N_G(P)/P}(Q) \cong C_{T_0}(Q) = \{1\}$ if and only if $Q$ contains a generating field automorphism.
\end{lem}

\begin{proof}
Rephrasing, we have $C_{T_0}(Q) = \{t \in T_0 \mid \varphi(t) = t$ for all $\varphi \in Q\}$. Now assume $\varphi \in Q$.  Then, as above, we write $\varphi$ as a product $\alpha\beta\gamma\delta$. Firstly, by definition, we have $\delta$ acts trivially on $T_0$ so we may assume $\varphi = \alpha\beta\gamma$. By assumption $\varphi$ stabilises $B_0$ and $T_0$, c.f., \cref{pa:assump-sylow}, which implies that $\alpha$ stabilises $B_0$ and $T_0$ because $\beta$ and $\gamma$ do by definition. This implies that $\alpha$ is affected by an element of $B_0$ because $N_G(B_0) = B_0$. As $\bT_0$ is non-degenerate we have by \cite[3.6.7]{carter:1993:finite-groups-of-lie-type} that $N_G(T_0) = N_{\bG}(\bT_0)^F$ so $N_{B_0}(T_0) = T_0$, hence $\alpha$ acts trivially on $T_0$. We may thus assume that $\varphi = \beta\gamma$.

Assume $F$ is twisted, so that $\varphi = \beta$ and for any $t \in T_0$ we have $\varphi(t) = t^{2^a}$ for some $a \in \mathbb{N}$. Identifying $T_0$ with a direct product $\mathbb{F}_{q^{m_1}}^{\times}\times\cdots\times \mathbb{F}_{q^{m_k}}^{\times}$ we see that $\varphi$ has a non-trivial fixed point if and only if $a > 1$. Now assume $F$ is split.  Then we may identify $T_0$ with a direct product $\mathbb{F}_q^{\times} \times \cdots \times \mathbb{F}_q^{\times}$ such that $\gamma$ permutes factors and $\beta$ acts as a $2$-power map. If $\gamma$ is non-trivial then $\varphi$ will have a non-trivial fixed point, so we are reduced to the previous case.
\end{proof}

\begin{prop}\label{prop:conjIF-p=2}
Assume $G$ is perfect, so the quotient $G/Z$ is simple.  Then any quasisimple group whose simple quotient is isomorphic to $G/Z$ satisfies \cref{cond:conjIF}.
\end{prop}

\begin{proof}
We will assume that $G$ has a trivial Schur multiplier because the remaining cases were dealt with in \cite[\S4]{schaeffer-fry:2015:odd-degree-characters} using explicit computations with {\sf GAP}. This means $G$ is a Schur cover and it suffices to show that \cref{cond:conjIF} holds for $G$, c.f., \cref{pa:conditions-remarks}.

We start with the assumption that the maximal torus $\bT_0$ is non-degenerate. Let $Q$ and $P \in \syl_2(G)$ be as in \cref{pa:assump-sylow} such that $C_{N_G(P)/P}(Q)=\{1\}$.  Then $Q$ contains a generating field automorphism $\varphi$ by \cref{lem:Q-centraliser}. We will denote by $\chi \in \mathcal{E}(G,s)$ a $Q$-invariant character of odd degree.

The bijective morphism $\varphi$ may be extended to a bijective morphism $\widetilde{\varphi} : \widetilde{\bG} \to \widetilde{\bG}$ by setting $\widetilde{\varphi}(z) = z^2$ for any $z \in Z(\widetilde{\bG})$. There then exists a dual bijective morphism $\varphi^{\star} : \widetilde{\bG}^{\star} \to \widetilde{\bG}^{\star}$ such that $F^{\star}\circ \widetilde{\varphi}^{\star} = \widetilde{\varphi}^{\star}\circ F^{\star}$ and $\widetilde{\varphi}^{\star}(t) = t^2$ for all $t \in \widetilde{\bT}_0^{\star}$. This map also descends to a homomorphism $\varphi^{\star} : \bT_0^{\star} \to \bT_0^{\star}$ defined by $\varphi^{\star}(t) = \iota^{\star}(\widetilde{\varphi}^{\star}(\widetilde{t}))$ where $\widetilde{t} \in \widetilde{\bT}_0^{\star}$ satisfies $\iota^{\star}(\widetilde{t}) = t$. This map is well defined because $\Ker(\iota^{\star}) = Z(\widetilde{\bG}^{\star})$, which is preserved by $\widetilde{\varphi}^{\star}$.

As the quotient $\widetilde{G}/G$ is an abelian, hence solvable, $2'$-group and $\langle \widetilde{\varphi}\rangle \leqslant \Aut(\widetilde{G})$ is a $2$-group we have by Glauberman's Lemma \cite[13.28]{isaacs:2006:character-theory-of-finite-groups} that there exists a character $\widetilde{\chi} \in \Irr(\widetilde{G}|\chi)$ covering $\chi$ which is fixed by $\widetilde{\varphi}$. Now, if $\widetilde{\chi}$ is contained in the Lusztig series $\mathcal{E}(\widetilde{G},\widetilde{s})$ then it is also contained in the Lusztig series $\mathcal{E}(\widetilde{G},\widetilde{\varphi}^{\star}(\widetilde{s}))$ by \cite[2.4]{navarro-tiep-turull:2008:brauer-characters-with-cyclotomic}. This implies $\widetilde{s}$ and $\varphi^{\star}(\widetilde{s})$ are $\widetilde{G}^{\star}$-conjugate.

There exists an element $g \in \widetilde{\bG}^{\star}$ such that ${}^g\widetilde{s} \in \widetilde{\bT}_0^{\star}$ so ${}^{g^{-1}\widetilde{\varphi}^{\star}(g)}\widetilde{\varphi}^{\star}(\widetilde{s}) = \widetilde{s}^2$, which means that $\widetilde{\varphi}^{\star}(\widetilde{s})$ is $\widetilde{\bG}^{\star}$-conjugate to $\widetilde{s}^2$. However, $\widetilde{\bG}^{\star}$-conjugacy is equivalent to $\widetilde{G}^{\star}$-conjugacy so this implies that $\widetilde{s}$ is $\widetilde{G}^{\star}$-conjugate to $\widetilde{s}^2$. By \cite[11.7]{bonnafe:2006:sln} we have $\chi \in \mathcal{E}(G,s)$ where $s = \iota^{\star}(\widetilde{s})$. Clearly we have $s$ is $G^{\star}$-conjugate to $s^2$ so every odd degree character in $\mathcal{E}(G,s)$ is $\sigma$-fixed by \cref{prop:odd-degree-fixed}. This shows that \cref{cond:conjIF} holds in this case.

Now consider the case where the maximal torus $\bT_0$ is degenerate. By \cref{lem:non-degen-tori} we have $q = 2$ but $G$ is not ${}^2\A_n(2)$ with $n \geqslant 2$. As we assumed that $G$ has a trivial Schur multiplier we have $G$ is not ${}^2\E_6(2)$ and so $G$ has a trivial centre. This implies $G \cong G^{\star}$ is a finite simple group so the argument in \cite[Lemma 2.4]{guralnick-malle-navarro:2004:self-normalising-sylows} shows that every semisimple element $s \in G^{\star}$ is $G^{\star}$-conjugate to $s^2$. By \cref{lem:galoisLseries,prop:odd-degree-fixed} we thus have every odd degree character of $G$ is $\sigma$-fixed so \cref{cond:conjIF} holds in this case.
\end{proof}

\section{\cref{cond:conjFI} when $p=2$}\label{sec:condition-conjFI}
\begin{pa}
Assume $G$ is perfect with centre $Z$, so the quotient $S = G/Z$ is simple. We now wish to outline a strategy for showing that $S$ satisfies \cref{cond:conjFI}. Firstly, we note that the homomorphism $\Aut(G) \to \Aut(S)$ induced by the natural surjection $G \to S$ is an isomorphism, see \cite[Theorem 2.5.14(d)]{gorenstein-lyons-solomon:1998:classification-3}. Now assume $A = SQ = GQ/Z$ for some $2$-group $Q \leqslant \Aut(S) \cong \Aut(G)$. We wish to show that if $A$ does not have a self-normalising Sylow $2$-subgroup, then there exists a character $\chi\in\irr_{2'}(S)$ which is $A$-invariant but is not fixed by $\sigma$. We will construct such a character by finding a character $\widetilde{\chi} \in \Irr_{2'}(\widetilde{G})$ such that $\widetilde{\chi}^\sigma \neq \widetilde{\chi}$ and the restriction $\chi=\Res_G^{\widetilde{G}}(\widetilde{\chi}) \in \Irr(G)$ is irreducible, $Q$-invariant, and has $Z$ in its kernel. We're then done by viewing $\chi$ as a character of $S$.
\end{pa}

\begin{pa}
Let $s$ be a semisimple element of $\widetilde{G}^\star$ then there exists a unique semisimple character $\widetilde{\chi}_s \in \mathcal{E}(\widetilde{G}, s)$ of $\widetilde{G}$, which has degree $\widetilde{\chi}_s(1) = [\widetilde{G}^{\star}:C_{\widetilde{G}^{\star}}(s)]_{p'}$. Recall that the number of irreducible constituents of $\chi := \Res_G^{\widetilde{G}}(\widetilde{\chi}_s)$ is exactly the number of irreducible characters $\theta \in \Irr(\widetilde{G}/G)$ satisfying $\widetilde{\chi}_s\theta = \widetilde{\chi}_s$. Furthermore, we have $\irr(\widetilde{G}/G) = \{\widetilde{\chi}_t \mid t \in Z(\widetilde{G}^{\star})\}$ and $\mathcal{E}(\widetilde{G}, s)\widetilde{\chi}_t = \mathcal{E}(\widetilde{G}, st)$ for such $t \in Z(\widetilde{G}^{\star})$, see \cite[13.30]{digne-michel:1991:representations-of-finite-groups-of-lie-type}. Hence we see that $\chi$ is irreducible if and only if $s$ is not $\widetilde{G}^\star$-conjugate to $st$ for any nontrivial $t\in Z(\widetilde{G}^\star)$. Moreover, if $s \in [\widetilde{G}^\star, \widetilde{G}^\star]$ then $\widetilde{\chi}_s$ is trivial on $Z(\widetilde{G})$ so $\Res^{\widetilde{G}}_G(\widetilde{\chi}_s)$ is trivial on $Z(G)$, see \cite[Lemma 4.4(ii)]{navarro-tiep:2013:characters-of-relative-pprime-degree}. Note that, by construction, the character $\chi$ is fixed by all inner and diagonal automorphisms of $G$.
\end{pa}

\begin{pa}\label{pa:s/s-inv}
Assume now that $s$ has odd order, so by \cref{lem:galoisLseries} and \cite[Corollary 2.4]{navarro-tiep-turull:2008:brauer-characters-with-cyclotomic}, we have $\chi_s^\sigma=\chi_{s^2}$ and $\chi_s^\psi=\chi_{\psi^{\star}(s)}$ for any $\psi \in \Aut(G)$, where $\psi^{\star} : \widetilde{G}^{\star} \to \widetilde{G}^{\star}$ is an automorphism dual to $\psi$. Hence to prove that \cref{cond:conjFI} holds it suffices to find an element $s \in [\widetilde{G}^\star, \widetilde{G}^\star]$ such that the following hold:
\begin{enumerate}[label=(S\arabic*)]
	\item $s$ has odd order and $[\widetilde{G}:C_{\widetilde{G}}(s)]_{p'}$ is odd,\label{it:S1}
	\item $s$ is not $\widetilde{G}^{\star}$-conjugate to $s^2$,\label{it:S2}
	\item $s$ is not $\widetilde{G}^{\star}$-conjugate to $st$ for any $t \in Z(\widetilde{G}^{\star})$,\label{it:S3}
	\item $s$ is $G^{\star}$-conjugate to $\psi^{\star}(s)$ for any field or graph automorphism $\psi \in Q$.\label{it:S4}
\end{enumerate}
With this in place we may now complete the proof of \Cref{thm:SN2SgoodTypeA} when $p=2$. Indeed, we have already shown that \cref{cond:conjIF} holds in \cref{prop:conjIF-p=2} so it suffices to show that \cref{cond:conjFI} holds under this assumption.
\end{pa}

\begin{prop}\label{prop:p-2-conjFI}
Assume $p=2$. If $G$ is perfect, then the finite simple group $S = G/Z$ satisfies \cref{cond:conjFI}.
\end{prop}

\begin{proof}
Assume $q = 2$. If $F$ is split then $T_0 = \{1\}$ so $G \cong S$ and $N_S(P)/P \cong T_0 = \{1\}$, so certainly \cref{cond:conjFI} holds in this case. The cases ${}^2\A_n(2)$, ${}^2\D_n(2)$, ${}^3\D_4(2)$, and ${}^2\E_6(2)$ are dealt with in \cite{schaeffer-fry:2015:odd-degree-characters} so we may assume that $q > 2$. We will now prove the statement by finding a semisimple element $s \in [\widetilde{G}^{\star},\widetilde{G}^{\star}]$ satisfying the conditions outlined in \cref{pa:s/s-inv}. What follows is a synthesised version of the arguments in \cite{schaeffer-fry:2015:odd-degree-characters}.

We will denote by $\widecheck{\Phi}^{\star} \subseteq \widecheck{X}(\widetilde{\bT}_0^{\star})$ the coroots of $\widetilde{\bG}^{\star}$ with respect to $\widetilde{\bT}_0^{\star}$. Clearly for any $\widecheck{\alpha} \in \widecheck{\Phi}^{\star}$ and $\zeta \in \mathbb{K}^{\times}$ we have $\widecheck{\alpha}(\zeta) \in [\widetilde{\bG}^\star, \widetilde{\bG}^\star]$. We now choose a set of simple coroots $\widecheck{\Delta}^{\star} = \{\widecheck{\alpha}_1,\dots,\widecheck{\alpha}_n\} \subseteq \widecheck{\Phi}^{\star}$, which corresponds to choosing a Borel subgroup of $\widetilde{\bG}^{\star}$ containing $\widetilde{\bT}_0^{\star}$. With this in place we fix a coroot
\begin{equation*}
\widecheck{\alpha}_0 = \widecheck{\alpha}_1+\cdots+\widecheck{\alpha}_n \in \widecheck{\Phi}.
\end{equation*}
Note this is always a coroot for any indecomposable root system, as is easily checked.

As $p=2$ we have $\Ker(\widecheck{\alpha}) = \{1\}$ for any coroot $\widecheck{\alpha} \in \widecheck{\Phi}^{\star}$, c.f., the proof of \cite[7.3.5]{springer:2009:linear-algebraic-groups}. In particular, the map $\mathbb{K}^{\times} \times \cdots \times \mathbb{K}^{\times} \to \widetilde{\bT}_0^{\star}$ defined by
\begin{equation}\label{eq:torus-embed}
(\zeta_1,\dots,\zeta_n) \mapsto \widecheck{\alpha}_1(\zeta_1)\cdots \widecheck{\alpha}_n(\zeta_n)
\end{equation}
is an injective morphism of algebraic groups. The torus $\bT_0$ is non-degenerate because we assume $q > 2$, c.f., \cref{lem:non-degen-tori}, so $A$ has a self-normalising Sylow $2$-subgroup if and only if $Q$ contains a generating field automorphism which we denote by $\varphi$, c.f., \cref{lem:Q-centraliser}. Let us write $q = p^a$ for some integer $a \geqslant 1$. By assumption, $Q$ is a $2$-group so it may contain any field automorphism of the form $\varphi^i$ where $i>1$ is a divisor of $a$ such that $a/i$ is a $2$-power.

With this in mind let us write $a = 2^tm$ with $t \geqslant 0$ and $m \geqslant 1$ odd. We define an automorphism
\begin{equation*}
\psi = \begin{cases}
\varphi^m &\text{if }m > 1,\\
\varphi^2 &\text{if }m = 1.
\end{cases}
\end{equation*}
Note that $\psi$ generates the subgroup of all field automorphisms that may possibly be contained in $Q$. For the moment we will assume that $q > 4$. We now fix an element $\zeta_0 \in \mathbb{K}^{\times}$ with the following properties:
\begin{enumerate}[label=(\roman*)]
	\item if $m > 1$ then $\zeta_0^2 \neq \zeta_0^{-1}$ and $\zeta_0^{2^m-1} = 1$,
	\item if $m=1$ then $\zeta_0 \in \mathbb{K}^{\times}$ is an element of order $5$.
\end{enumerate}
Now consider the corresponding element $s_0 = \widecheck{\alpha}_0(\zeta_0) \in \widetilde{\bT}_0^{\star}$. One readily checks that if $q > 4$ then the element $s_0$ is $F^{\star}$-fixed. Now assume $q = 4$ and denote by $\dot{w} \in N_{\widetilde{\bG}^{\star}}(\widetilde{\bT}_0^{\star})$ an element representing the reflection of $\widecheck{\alpha}_0$. If $g \in \widetilde{\bG}^{\star}$ is an element such that $g^{-1}F^{\star}(g) = \dot{w}$ then clearly the conjugate $s = {}^gs_0$ is $F^{\star}$-fixed. Hence, in all cases we have defined a rational semisimple element $s \in \widetilde{T}_0^{\star}$ contained in the derived subgroup $[\widetilde{G}^{\star},\widetilde{G}^{\star}]$. We now show that the conditions \crefrange{it:S1}{it:S4} hold for $s$.

\cref{it:S1}. As $p=2$ this clearly holds for any semisimple element.

\cref{it:S2}. We claim that $s$ and $s^2$ are not $\widetilde{\bG}^{\star}$-conjugate, hence are not $\widetilde{G}^{\star}$-conjugate. If they were $\widetilde{\bG}^{\star}$-conjugate then $s_0$ would be $\widetilde{\bG}^{\star}$-conjugate to $s_0^2$ so by \cite[3.7.1]{carter:1993:finite-groups-of-lie-type} there would exist an element $\dot{w} \in N_{\widetilde{\bG}^{\star}}(\widetilde{\bT}_0^{\star})$, representing $w \in W_{\widetilde{\bG}^{\star}}(\widetilde{\bT}_0^{\star}) := N_{\widetilde{\bG}^{\star}}(\widetilde{\bT}_0^{\star})/\widetilde{\bT}_0^{\star}$, such that ${}^{\dot{w}}s_0 = s_0^2$. Assume $w(\widecheck{\alpha}_0) = a_1\widecheck{\alpha}_1+\cdots+a_n\widecheck{\alpha}_n$ with $a_i \in \mathbb{Z}$ then ${}^{\dot{w}}s_0 = \widecheck{\alpha}_1(\zeta_0^{a_1})\cdots \widecheck{\alpha}_n(\zeta_0^{a_n})$. Clearly $w(\widecheck{\alpha}_0) \in \widecheck{\Phi}^{\star}$ is a coroot. Inspecting the indecomposable root systems one easily observes that one of the following is true: $a_i = \pm 1$ or $(a_i,a_j) = (\pm 2, \pm 3)$ for some $1 \leqslant i,j \leqslant n$. In particular, the condition ${}^{\dot{w}}s_0 = s_0^2$ implies that either $\zeta_0^2 = \zeta_0^{\pm 1}$ or $\zeta_0^2 = \zeta_0^{\pm 2} = \zeta_0^{\pm 3}$. From the choice of our element $\zeta_0$ one easily confirms that this is impossible, so $s$ cannot be $\widetilde{G}^{\star}$-conjugate to $s^2$.

\cref{it:S3}. We need only show that $C_{\bG^{\star}}(s) = {}^gC_{\bG^{\star}}(s_0)$ is connected, see \cite[2.8(a)]{bonnafe:2005:quasi-isolated}. The argument used above shows that an element $\dot{w} \in N_{\bG^{\star}}(\bT_0^{\star})$, representing $w \in W_{\bG^{\star}}(\bT_0^{\star})$, satisfies ${}^{\dot{w}}s_0 = s_0$ if and only if $w(\widecheck{\alpha}_0) = \widecheck{\alpha}_0$. The centraliser of $\widecheck{\alpha}_0$ in $W_{\bG^{\star}}(\bT_0^{\star})$ is a parabolic subgroup, see \cite[A.29]{malle-testerman:2011:linear-algebraic-groups}, which implies that $C_{\bG^{\star}}(s_0)$ is connected by \cite[2.4]{digne-michel:1991:representations-of-finite-groups-of-lie-type}. We thus have $C_{\bG^{\star}}(s)$ is connected.

\cref{it:S4}. Assume $\gamma \in Q$ is a graph or field automorphism. As $C_{\bG^{\star}}(s)$ is connected we have $\gamma^{\star}(s)$ is $G^{\star}$-conjugate to $s$ if and only $\gamma^{\star}(s)$ is $\bG^{\star}$-conjugate to $s$. Moreover, it is clear that $\gamma^{\star}(s)$ is $\bG^{\star}$-conjugate to $s$ if and only $\gamma^{\star}(s_0)$ is $\bG^{\star}$-conjugate to $s_0$. Now certainly $s_0$ is fixed by all graph automorphisms. If $m > 1$ then we have $\psi^{\star}(s_0) = s_0^{2^m} = s_0$ and if $m = 1$ then we have $\psi^{\star}(s_0) = s_0^4 = s_0^{-1}$. However if $\dot{w} \in N_{\bG^{\star}}(\bT_0^{\star})$ represents the reflection of $\widecheck{\alpha}_0$ then ${}^{\dot{w}}s_0 = s_0^{-1}$ so we're done.
\end{proof}

\section{Sylow $2$-Subgroups of $\GL_n^{\varepsilon}(q)$}\label{sec:sylow}
\begin{assumption}
From this point forward we assume that $p$ is odd, $\bG = \SL_n(\mathbb{K})$, $\widetilde{\bG} = \GL_n(\mathbb{K})$, and $\iota$ is the natural inclusion map. Moreover, we assume that $\bG^{\star} = \PGL_n(\mathbb{K})$, $\widetilde{\bG}^{\star} = \GL_n(\mathbb{K})$, and $\iota^{\star}$ is the natural projection. The Frobenius endomorphism $F$ will be assumed to denote either the morphism $F_q$ or $F_q\phi = \phi F_q$, with the notation as in the proof of \cref{lem:non-degen-tori}. The reference tori $\bT_0$, $\bT_0^{\star}$, $\widetilde{\bT}_0$, $\widetilde{\bT}_0^{\star}$ will be taken to be the maximal tori of diagonal matrices.
\end{assumption}

\begin{pa}
Throughout we will adopt the following convention: The split group $\bG^{F_q}$, resp., twisted group $\bG^{F_q\phi}$, which we continue to refer to as $G$, will be denoted by $\SL_n^{+1}(q)$, resp., $\SL_n^{-1}(q)$. To unify this we let $\varepsilon$ denote $\pm1$ and simply write $\SL_n^{\varepsilon}(q)$ to denote the two rational forms of $\SL_n(\mathbb{K})$. We also write $\GL_n^{\varepsilon}(q)$ and $\PGL_n^{\varepsilon}(q)$ to have the corresponding meanings. Furthermore we define $\overline{q}$ to be $q$ if $\varepsilon = 1$ and $q^2$ if $\varepsilon = -1$. With this we have natural embeddings $\SL_n^{\varepsilon}(q) \leqslant \SL_n(\overline{q})$, $\GL_n^{\varepsilon}(q) \leqslant \GL_n(\overline{q})$, and $\PGL_n^{\varepsilon}(q) \leqslant \PGL_n(\overline{q})$.  Recall that in this setting, $\widetilde{G} = \widetilde{G}^\star = \GL_n^\varepsilon(q)$, $G^\star = \PGL^\varepsilon_n(q)$, and we write $Z$ for the centre $Z(G)$ of $\SL_n^\varepsilon(q)$.
\end{pa}

\begin{pa}
In this section we recall results of Carter--Fong on the Sylow $2$-subgroups of $\GL_n^\varepsilon(q)$. For this we introduce the following notation.  For $r \geqslant 0$ an integer, we denote by $S_r^{\varepsilon}(q)$ a Sylow $2$-subgroup of $\GL_{2^r}^{\varepsilon}(q)$. With this in place we have the following, see \cite[Theorem 1, Theorem 4]{carter-fong:1964:the-Sylow-2-subgroups}.
\end{pa}

\begin{thm}[Carter--Fong]\label{thm:carter-fong}
Let $n = 2^{r_1} + \cdots + 2^{r_t}$, with $0\leqslant r_1 < \cdots < r_t$,  be an integer written in its $2$-adic expansion. If $\widetilde{P} \leqslant \widetilde{G}$ is a Sylow $2$-subgroup of $\widetilde{G} = \GL_n^{\varepsilon}(q)$ then $\widetilde{P} \cong \prod_{i = 1}^t S_{r_i}^{\varepsilon}(q)$ and
\begin{equation}\label{CF2}
N_{\wt{G}}(\wt{P}) \cong \wt{P} \times C_{(q-\epsilon)_{2'}}\times\cdots\times C_{(q-\epsilon)_{2'}}
\end{equation}
with $t$ copies of the cyclic group $C_{(q-\epsilon)_{2'}}$.
\end{thm}

\begin{pa}
The group $N_{\wt{G}}(\wt{P})$ can be described more explicitly. Firstly, the Sylow $\widetilde{P}$ can be realised by embedding $\prod_{i = 1}^t S_{r_i}^{\varepsilon}(q) \leqslant \prod_{i = 1}^t \GL_{2^{r_i}}^{\varepsilon}(q)$ block-diagonally in a natural way. Now for each $1\leq j\leq t$ the corresponding factor $C_{(q-\epsilon)_{2'}}$ is embedded as the largest odd-order subgroup of the centre $Z(\GL^\varepsilon_{2^{r_j}}(q))$. In particular, writing $I_k$ for the identity of $\GL_k(q)$, elements of $N_{\wt{G}}(\wt{P})$ are of the form $xz$ where $x\in\wt{P}$ and
\begin{equation}\label{CF3}
z = \bigoplus_{i=1}^t \lambda_j I_{2^{r_j}} = \diag(\lambda_1 I_{2^{r_1}},\dots,\lambda_t I_{2^{r_t}})
\end{equation}
with $\lambda_j \in C_{(q-\epsilon)_{2'}}\leq \mathbb{F}_{\overline{q}}^\times$. In what follows, we will use the notation $z=\bigoplus_{j=1}^t z_j$ for this matrix with $z_j = \lambda_j I_{2^{r_j}}$ for each $1\leq j\leq t$. We close this section with a result which will be used as part of the proof of \cref{prop:extension-to-GLn}.
\end{pa}

\begin{lem}\label{cor:index-GL}
Let $m = 2^{r_1}+\cdots+2^{r_t} \in \mathbb{N}$, with $0 \leqslant r_1 < \cdots < r_t$, be an integer written in its $2$-adic expansion.  Then
\begin{equation*}
[\GL_{2m}^{\varepsilon}(q) : \GL_m^{\varepsilon}(q)^2]_2 = 2^t
\end{equation*}
where $n_2$ denotes the $2$-part of an integer $n \geqslant 1$.
\end{lem}

\begin{proof}
As $2m = 2^{r_1+1} + \cdots + 2^{r_t+1}$ is clearly the $2$-adic expansion of $2m$, we have by \cref{thm:carter-fong} that
\begin{equation*}
[\GL_{2m}^{\varepsilon}(q) : \GL_m^{\varepsilon}(q)^2]_2 = \prod_{i = 1}^t \left(|S_{r_i+1}^{\varepsilon}(q)|/|S_{r_i}^{\varepsilon}(q)|^2\right).
\end{equation*}
According to \cite[Eq.\ (4)]{carter-fong:1964:the-Sylow-2-subgroups} we have $|S_{r+1}^{\varepsilon}(q)| = 2|S_r^{\varepsilon}(q)|^2$ for any integer $r \geqslant 0$. From this the result follows immediately.
\end{proof}

\section{\cref{cond:conjFI} for Type $\A$}
\begin{pa}
Let $\wt{P}$ be a Sylow $2$-subgroup of $\wt{G}$, so that $P=\wt{P}\cap G$ is a Sylow $2$-subgroup of $G$ which is normal in $\wt{P}$.  Then \cite[Theorem 1]{kondratiev:2005:normalizers-of-sylow-2-subgroups} yields that
\begin{equation}\label{CF1}
N_{\wt{G}}(P) = \wt{P}C_{\wt{G}}(\wt{P}) = N_{\wt{G}}(\wt{P}),
\end{equation}
and hence we see that $N_G(P)=N_G(\wt{P})=N_{\wt{G}}(\wt{P})\cap G$. Now, if $n$ is not a power of $2$, write
\begin{equation}\label{2adic}
n=2^{r_1}+2^{r_2}+...+2^{r_t}
\end{equation}
with $t\geq 2$ and $r_1>r_2>...>r_t\geq 0$ for the $2$-adic expansion of $n$. We now wish to describe when the quotient $GQ/Z$, with $Q \leqslant \Aut(G)$ a $2$-group, has a self-normalising Sylow $2$-group; thus allowing us to show \cref{cond:conjFI} holds. The following gives a complete description of those subgroups $Q$ with this property.
\end{pa}

\begin{lemma}[see \cite{kondratiev:2005:normalizers-of-sylow-2-subgroups}]\label{lem:whenSN2Ssimple}
A simple group $\PSL_n^\epsilon(q)$ has a self-normalising Sylow $2$-subgroup if and only if one of the following holds:
\begin{itemize}
	\item[(i)] $n = 2^r$ for some $r\geqslant 2$,
	\item[(ii)] $ n \neq 2^r$ for any $r \geqslant 2$ and $(q-\epsilon)_{2'} = 1$,
	\item[(iii)] $n = 2^{r_1}+2^{r_2}$ for some $r_1 > r_2 \geqslant 0$ and $(q-\epsilon)_{2'} = (n,q-\epsilon)_{2'}$.
\end{itemize}
\end{lemma}

\begin{lemma}\label{lem:whenSN2S}
Write $\overline{q}=p^a$ and let $Q \leqslant \Aut(G)$ be a $2$-group. The quotient $GQ/Z$ has a self-normalising Sylow $2$-subgroup if and only if at least one of the following is satisfied:
\begin{itemize}
	\item[(1)] $G/Z$ has a self-normalising Sylow $2$-subgroup; 
	\item[(2)] $Q$ contains a graph automorphism in case $\epsilon = 1$ or an involutary field automorphism in case $\epsilon=-1$, either of which we may identify as the map $\phi$, up to inner and diagonal automorphisms;
	\item[(3)] $\epsilon=1$, $a$ is a $2$-power,  $(p-1)_{2'}=1$, and $Q$ contains a field automorphism of order $a$ (which we identify with $F_{ p}$, up to inner and diagonal automorphisms);
	\item[(4)] $\epsilon=1$, $a$ is a $2$-power, $p=3$, and $Q$ contains a field automorphism of order $a/2$ (which we identify with $F_3^2$, up to inner and diagonal automorphisms); or
	\item[(5)] $\epsilon=1$, $n=2^{r_1}+2^{r_2}$ for integers $r_1>r_2\geq 0$, $(p^m-1)_{2'}=\gcd(n,p^m-1)_{2'}$ for some $m$ dividing $a$, and $Q$ contains a field automorphism of order $a/m$ (which we identify with $F_{ p}^m$, up to inner and diagonal automorphisms).
\end{itemize}
\end{lemma}

\begin{rem}
Since the involutary field automorphism $F_q$ induces the map $\phi$ on $\GU_n(q)$, condition (2) in the case $\epsilon=-1$ includes the case that $Q$ contains any field automorphism whose order is a power of $2$.
\end{rem}

\begin{proof}[of \cref{lem:whenSN2S}]
Let $P$ be a Sylow $2$-subgroup of $G$ stabilized by $Q$.  Specifically, we may choose $P$ as in the setup for \eqref{CF1}.  

(I) First suppose that one of (1), (2), (3), (4), or (5) holds.  Note that in case (1), the statement is certainly true, since then $N_G(P)=PZ$, so $C_{N_G(P)/PZ}(Q)=1$.  Hence we may assume that $G/Z$ does not have a self-normalising Sylow $2$-subgroup and that $Q$ contains an outer automorphism.  Specifically, either $Q$ contains a graph automorphism (in case $\epsilon =1$) or involutary field automorphism (in case $\epsilon = -1$), which we identify with $\phi$ on $G$, up to conjugation in $\widetilde{G}$; or $\epsilon=1$ and $Q$ contains a field automorphism, which we identify as $F_{p^m}$ on $G$, up to conjugation in $\widetilde{G}$, for some $m\geq 1$.  Write $\varphi$ for the corresponding graph or field automorphism, respectively.  We will show that $C_{N_G(P)/PZ}(\varphi)=1$.  Write $\overline{N}:=N_G(P)/PZ$ and let $\overline{g}$ denote the image of an element $g\in N_G(P)$ in $\overline{N}$.  Suppose $g\in N_G(P)$ satisfies that $\overline{g}\in C_{\overline{N}}(\varphi)$.  That is, $\overline{g}$ is fixed by $\varphi$.  

 Write $n$ as in \cref{2adic}, so that by \cref{CF1,CF3} we have $g=xz$ for some $x\in P$ and $z=\bigoplus_{j=1}^t\lambda_jI_{2^{r_j}}$ as in \eqref{CF3} such that $\prod_{j=1}^t \lambda_j^{2^{r_j}} = 1$.  Then observing the action of $\varphi$ on the $2'$-part of $g$, we see $\varphi(z)=zy$ for some $y\in Z$ of odd order.  Write $y=\eta I_{n}$ for some $(n, q-\epsilon)$-root of unity $\eta$ in $\mathbb{F}_{\overline{q}}^\times$.  Then  since the block sizes $2^{r_j}$ are distinct, we must have that $\lambda_j\eta=\lambda_j^{-1}$ or   $\lambda_j^{p^m}$, respectively, for each $1\leq j\leq t$.  

(IA) Hence if condition (2), (3), or (4) holds, then there is some integer $c\geq 1$ such that $\lambda_j^{2^c}=\eta$ for each $1\leq j\leq t$.  (Recall that in situation (3) $p-1$ is a power of $2$, and in situation (4), $p^m=9$, so $p^m-1$ is also a power of $2$.)  Then in these cases, for each $j$ there is a $2^c$-root of unity $\zeta_j$ satisfying $\lambda_1=\zeta_j\lambda_j$, and we may write $z$ as the product of $\lambda_1 I_n$ and a diagonal matrix $d$ whose diagonal entries are $2$-power roots of unity.  Further, since $z$ has determinant $1$, $|d|$ has $2$-power order, and the multiplicative order of $\lambda_1$ is odd, it follows that $\lambda_1 I_n\in Z$ and $d\in P$.  We therefore see that $g\in PZ$, so that $\overline{g}=1$, yielding that in cases (2), (3), and (4), $C_{\overline{N}}(\varphi)=1$. 

(IB) Now assume condition (5) holds, so that $t=2$, $\epsilon=1$, and $(p^m-1)_{2'}=\gcd(n, p^m-1)_{2'}$.  Note then that $\PSL_n(p^m)$ and $\PGL_n(p^m)$ have self-normalising Sylow $2$-subgroups, see \cref{lem:whenSN2Ssimple}.  Note that $\lambda_1^{p^m-1}=\eta=\lambda_2^{p^m-1}$, so that $\lambda_1=\zeta\lambda_2$, for some $(p^m-1)$-root of unity $\zeta$ in $\mathbb{F}_q^\times$.  
 
 Then as an element of $\GL_n(q)$, we may write $z$ as the product of the central element $\lambda_1 I_n$ and a diagonal matrix $d$ whose diagonal entries are $(p^m-1)$-roots of unity.  In particular, $d\in \GL_n(p^m)$ is an element centralising a Sylow $2$-subgroup $\wt{P}_m$ contained in $\wt{P}$ (see the constructions in \cite{carter-fong:1964:the-Sylow-2-subgroups}), where $\wt{P}$ is a Sylow $2$-subgroup of $\GL_n(q)$ such that $P=\wt{P}\cap G$.  Hence the image of $z$ in $G/Z\cong GZ(\wt{G})/Z(\wt{G})$ must be trivial, since $\PGL_n(p^m)$ has a self-normalising Sylow $2$-subgroup and $z$ has odd order.  Then again, $\overline{g}=1$, yielding that $C_{\overline{N}}(\varphi)=1$ in case (5) as well.
 
 (II) Now, assume that none of (1) to (5) hold.   We will show that $C_{N_G(P)/PZ}(Q)\neq 1$, so that $GQ/Z$ does not have a self-normalising Sylow $2$-subgroup.   We do this by exhibiting a nontrivial element of $\overline{N}:=N_G(P)/PZ$ which is fixed by all possible elements of $Q$.  Note that we may assume $Q$ contains an outer automorphism.  
 
Since (1) does not hold, we see that neither $n$ nor $(q-\epsilon)$ is a power of $2$, see \cref{lem:whenSN2Ssimple}. Hence writing $n$ as in \eqref{2adic}, we see $t\geq 2$ and there exist nonidentity elements of the the form $z=\oplus_{j=1}^t \lambda_j I_{2^{r_j}}$ as in \eqref{CF3}.  Since (2) does not hold, $Q$ does not contain $\phi$ up to conjugation in $\wt{G}$.  Further, any diagonal automorphism in $Q$ is induced by the quotient group $\wt{P}/P$, and therefore is centralised by such a $z$ by construction.  Hence it suffices to exhibit a $z$ such that $\varphi(z)=z$ for each field automorphism $\varphi$ contained in $Q$, the $\lambda_j$ for $1\leq j\leq t$ are not all the same, and $\lambda_1^{2^{r_1}}\cdot...\cdot\lambda_t^{2^{r_t}}=1$. 
 
(IIA) Suppose that $Q$ contains no field automorphisms.  (In particular, this is the case if $\epsilon=-1$.)  Let $\lambda_j=1$ for $j>2$ and let $\lambda_2$ be a primitive $(q-\epsilon)_{2'}$ root of unity in $\mathbb{F}_{\overline{q}}^\times$ and $\lambda_1=\lambda_2^{b}$, where $b\equiv -2^{r_2-r_1}\pmod {(q-\epsilon)_{2'}}$. (Note that this is possible since $(q-\epsilon)_{2'}\neq1$ and  $2$ is invertible modulo $(q-\epsilon)_{2'}$.)  Then the determinant of $z$ is 
$$\lambda_1^{2^{r_1}}\cdot\lambda_2^{2^{r_2}}=\lambda_2^{2^{r_1}b}\cdot\lambda_2^{2^{r_2}}=\lambda_2^{-2^{r_1}(2^{r_2-r_1})}\cdot\lambda_2^{2^{r_2}}=\lambda_2^{-2^{r_2}+2^{r_2}}=1,$$ so that $z\in G$.  Further, if $t>2$, then the $\lambda_j$ are not all the same, so $z$ is not central, and the proof is complete in this case.  

If $t=2$, then since (1) does not hold, we know by \cref{lem:whenSN2Ssimple} that $\gcd(n,q-\epsilon)_{2'}\neq(q-\epsilon)_{2'}$, so $\gcd(2^{r_1-r_2}+1, q-\epsilon)_{2'}\neq(q-\epsilon)_{2'}$.  This yields that $b\not\equiv1\bmod{(q-\epsilon)_{2'}}$, so $\lambda_1\neq\lambda_2$, and $z$ is again not central.  

(IIB) Now assume that $\epsilon=1$ and that $Q$ contains a field automorphism.  Write $q=p^a$.  Without loss, we may identify the generator of the subgroup of $Q$ consisting of field automorphisms as $F_p^m$ for some $m|a$.   Further, it suffices to assume that $F_p^m$ generates the largest $2$-group of automorphisms possible without inducing conditions (3)-(5).  Note that since (5) does not hold, we have $(p^m-1)_{2'}\neq \gcd(n, p^m-1)_{2'}$ if $t=2$.
   
Note that if $(p-1)_{2'}\neq1$, then also $(p^m-1)_{2'}\neq 1$.  If $p-1$ and $a$ are both powers of $2$, then since neither (3) nor (4) hold, we may assume that $2|m$ when $p\neq 3$ and that $4|m$ when $p=3$.  Then $(p^m-1)_{2'}\neq 1$, since both of $p-1$ and $p+1$ cannot simultaneously be a power of 2 unless $p=3$, in which case $3^4-1$ is not a $2$-power.   If $p-1$ is a power of $2$ but $a$ is not, then we may assume that $m$ is divisible by $a_{2'}$, the odd part of $a$.  Then $(p^m-1)_{2'}$ is divisible by the $a_{2'}$`th cyclotomic polynomial evaluated at $p$, which is odd since $a_{2'}$ and $p$ are.   

Hence in all cases, we may assume $p^m-1$ is not a $2$-power.  Then repeating the argument in (IIA) with $q$ replaced with $p^m$, we may choose $\lambda_1, \lambda_2\in\mathbb{F}_{p^m}^\times$ so that $z$ is non-central and lies in $\SL_n(p^m)$.  Hence $z$ is fixed by the field automorphisms in $Q$ and has the required form. 
\end{proof}

\begin{proposition}\label{prop:typeAconjFI}
If $G$ is perfect then the simple group $G/Z = \PSL_n^\epsilon(q)$ satisfies \cref{cond:conjFI}.
\end{proposition}

\begin{proof}
By \cref{prop:p-2-conjFI} our assumption that $p$ is odd is not restrictive. We will argue this by proving the contrapositive, as in \cref{sec:condition-conjFI}. Specifically we assume $A = SQ = GQ/Z \leq\aut(S)$ is a group obtained by adjoining a $2$-group $Q$ of automorphisms to $S$. We wish to show that if $A$ does not have a self-normalizing Sylow $2$-subgroup, i.e., $Q$ is not as in (1) to (5) of \cref{lem:whenSN2S}, then there exists a character $\chi\in\irr_{2'}(S)$ which is $A$-invariant but not fixed by $\sigma$. We do this by finding a semisimple element $s \in [\widetilde{G}^{\star},\widetilde{G}^{\star}] = \SL_n^{\varepsilon}(q)$ satisfying the conditions in \cref{pa:s/s-inv}.

Since $A$ has no self-normalising Sylow $2$-subgroup, we see by \cref{lem:whenSN2Ssimple} that neither $n$ nor $q-\epsilon$ is a power of $2$.  Write
\begin{equation*}
n=2^{r_1}+2^{r_2}+...+2^{r_t}
\end{equation*}
with $t\geq 2$ and $r_1>r_2>...>r_t\geq 0$ for the $2$-adic expansion of $n$.  From the discussion in \cref{sec:sylow}, to ensure that $s$ has odd order and centralises a Sylow $2$-subgroup of $\wt{G}^\star$, it suffices to choose a nonidentity $s$ in the form $s=\bigoplus_{j=1}^t \lambda_j\cdot I_{2^{r_j}}$, as in \eqref{CF3}.

If $s$ is non-central, then since the block sizes $2^{r_j}$ are distinct, it follows that this choice of $s$ is not conjugate in $\GL_n^\epsilon(q)$ to $s^2$ or $st$ for any nontrivial $t\in Z(\GL_n^\epsilon(q))$.

We note further that $A$ does not contain a graph automorphism, by \cref{lem:whenSN2S}.  Hence it suffices to exhibit an $s$ as above such that: $\varphi(s)$ is conjugate to $s$ for each field automorphism $\varphi$ contained in $A$, the $\lambda_j$ for $1\leq j\leq t$ are not all the same, and $\lambda_1^{2^{r_1}}\cdot...\cdot\lambda_t^{2^{r_t}}=1$ so that $s\in \SL_n^\epsilon(q)$. 

Letting $s$ be the element $z$ obtained in parts (IIA) and (IIB) of the proof of \cref{lem:whenSN2S}, we see that the conjugacy class of $s$ is fixed by $A$ and has the required form.
\end{proof}

\section{Covering Odd Degree Characters of $\SL_n^{\varepsilon}(q)$}
\begin{pa}
We wish to show that $G$ satisfies the hypotheses of \cref{cond:conjIF}. As we already saw in the proof of \cref{prop:conjIF-p=2} it is important to know that a $\sigma$-invariant odd degree character of $G$ can be covered by a $\sigma$-invariant character of $\widetilde{G}$. Unfortunately, we cannot appeal to Glauberman's Lemma as in the proof of \cref{prop:conjIF-p=2}. The following gives the desired covering result.
\end{pa}

\begin{proposition}\label{prop:seriesabove}
Let $S = G/Z$ and suppose $Q \leqslant \Aut(S)$ is a $2$-group such that $GQ/Z$ has a self-normalising Sylow $2$-subgroup. Assume $\lambda\in\irr(Z)$ is $\sigma$-fixed and $Q$-invariant and let $\chi\in\irr_{2'}(G|\lambda)$ be $Q$-invariant. Then there exists an irreducible character $\widetilde{\chi} \in \Irr(\widetilde{G}|\chi)$ covering $\chi$ which is contained in a Lusztig series $\mathcal{E}(\widetilde{G},\widetilde{s})$ labelled by an element $\widetilde{s}$ of $2$-power order. In particular, we have $\widetilde{\chi}^{\sigma} = \widetilde{\chi}$.
\end{proposition}

\begin{proof}
Let $\chi$ be as in the statement, so $\chi \in \irr(G)$ has odd degree. Then, in particular, $\chi$ lies in a series $\mathcal{E}(G, s)$ for some semisimple element $s \in G^\star$ for which $[G^\star : C_{G^\star}(s)]_{p'}$ is odd. This implies $s$ centralises, hence normalises, a Sylow $2$-subgroup of $G^\star$.

Now, the characters $\wt{\chi} \in \irr(\widetilde{G}|\chi)$ lying above $\chi$ are members of rational series of the form $\mathcal{E}(\widetilde{G},\widetilde{s})$, where $\widetilde{s} \in \widetilde{G}^\star$ satisfies $\iota^{\star}(\widetilde{s}) = s$, see \cite[Corollaire 9.7]{bonnafe:2006:sln}. We will write $\widetilde{Z}$ for the centre $Z(\widetilde{G}^\star)$ and denote by $\widetilde{P}$ a Sylow $2$-subgroup of $\widetilde{G}^\star$ such that $s$ centralises $\widetilde{P}\wt{Z}/\wt{Z}$. We aim to show that $\widetilde{s}$ may be chosen to have $2$-power order. If this is the case then by \cref{lem:galoisLseries,prop:GLn-sigma-fixed} we must have $\widetilde{\chi}^{\sigma} = \widetilde{\chi}$.

First, suppose that $Q$ is as in \cref{lem:whenSN2S}(1), so that $S$ has a self-normalising Sylow $2$-subgroup.  Then $\PGL^\epsilon_n(q)$ also has a self-normalising Sylow $2$-subgroup, so $s$ must be contained in the Sylow $2$-subgroup $\widetilde{P}\wt{Z}/\wt{Z}$ of $\PGL^\epsilon_n(q)$. Let $\wt{s}'=rz$ be a pre-image of $s$, where $r\in\wt{P}$ and $z\in\wt{Z}$. Then noting that $\wt{s}=\wt{s}'z^{-1}$ is another pre image of $s$, the claim is proved in this case.

Next, assume condition (2), (3), (4), or (5) of \cref{lem:whenSN2S} holds.  Then either $Q$ contains a graph automorphism (in case $\epsilon = 1$) or involutary field automorphism (in case $\epsilon = -1$), which we identify with $\phi$ on $\widetilde{G}\cong \widetilde{G}^\star\cong \GL^\epsilon_n(q)$; or $\epsilon=1$ and $Q$ contains a field automorphism, which we identify as $F_{p^m}$ on $\widetilde{G}\cong \widetilde{G}^\star\cong \GL_n(q)$, for some $m\geq 1$. By an abuse of notation, write $\varphi$ for $\phi$ or $F_{p^m}$, respectively, on $\widetilde{G}$ and $\widetilde{G}^\star$.  Then $\chi^\varphi=\chi$, and in particular, $\mathcal{E}(G,s)^\varphi=\mathcal{E}(G,s)$, yielding that the class $(s)$ is fixed by $\varphi$, by \cite[Corollary 2.4]{navarro-tiep-turull:2008:brauer-characters-with-cyclotomic}.

Let $\wt{s}\in \GL^\epsilon_n(q)$ be a pre-image of $s$.     Then if $x\in\wt{P}$, we see $\wt{s}x\wt{s}^{-1}\in x \wt{Z}$. But further, $\wt{s}x\wt{s}^{-1}$ has order a power of $2$, so must be contained in the unique Sylow $2$-subgroup $\wt{P}$ of $\wt{P}\wt{Z}$.  We therefore see that $\wt{s}$ normalises $\wt{P}$.  

Write $n=2^{r_1}+2^{r_2}+...+2^{r_t}$ with $r_1>r_2>...>r_t$ for the $2$-adic expansion of $n$.  From the discussion in \cref{sec:sylow}, we may then choose $\wt{s}$ in the form $\wt{s}=s_2z$, where $s_2\in \wt{P}$ and $z=\bigoplus_{j=1}^t \lambda_j\cdot I_{2^{r_j}}$, as in \eqref{CF3}.

Further, since $\varphi(s)$ is conjugate in $\PGL^\epsilon_n(q)$ to $s$, we see that $\varphi(\tilde{s})$ is conjugate in $\GL^\epsilon_n(q)$ to $\wt{s}y$ for some $y\in\wt{Z}$.  Then $ \varphi(z)$ is conjugate to $z y_{2'}$, where $y_{2'}$ denotes the odd part of $y$.  Let $\eta\in \mathbb{F}_{\overline{q}}^\times$ be such that $y_{2'}=\eta I_{n}$. Then as in part (I) of the proof of \cref{lem:whenSN2S}, since the block sizes $2^{r_j}$ are distinct, we must have that these eigenvalues satisfy $\lambda_j\eta=\lambda_j^{-1}$ or   $\lambda_j^{p^m}$, respectively, for each $1\leq j\leq t$.  

Hence if condition (2), (3), or (4) holds, arguing exactly as in part (IA) of the proof of \cref{lem:whenSN2S} yields that we may write $z$ as the product of $\lambda_1I_n\in\wt{Z}$ and a diagonal matrix whose diagonal entries are $2$-power roots of unity.  We may then replace $\wt{s}$ with $\wt{s}\lambda_1^{-1}$, which is also a pre-image of $s$ and has $2$-power order, completing the proof in this case.

Finally, assume condition (5) of \cref{lem:whenSN2S} holds, so that $t=2$, $\epsilon=1$, and $(p^m-1)_{2'}=\gcd(n, p^m-1)_{2'}$.  Note then that $\PSL_n(p^m)$ and $\PGL_n(p^m)$ have self-normalising Sylow $2$-subgroups.  Further, by multiplying by the central element $\lambda_2^{-1}I_n$, we may assume that $z=\lambda_1 I_{2^{r_1}}\oplus I_{2^{r_2}}$, and $\eta=1$.  Therefore it must be that $\lambda_1^{p^m-1}=1$, yielding that $\lambda_1 I_{2^{r_1}}$ is an element of order dividing $(p^m-1)_{2'}$ in $Z(\GL_{2^{r_1}}(p^m))$.  That is, $z$ is contained in the centraliser in $\GL_n(p^m)$ of a Sylow $2$-subgroup $\wt{P}_m$ contained in $\wt{P}$ (see the constructions in \cite{carter-fong:1964:the-Sylow-2-subgroups}).  Then the image of $z$ in $\PGL_n(q)$ is an element of odd order in $\PGL_n(p^m)$ centralising $\wt{P}_m\wt{Z}/\wt{Z}$, which is self-normalising in $\PGL_n(p^m)$.  This yields that the image of $z$ in $\PGL_n(q)$ is trivial, so $s$ is a $2$-element.  Arguing as in the case that \cref{lem:whenSN2S}(1) holds, the proof is complete.
\end{proof}

\section{\cref{cond:conjIF} for Type $\A$}
\begin{pa}
We wish to understand the effect of the Galois automorphism $\sigma$ on the odd degree characters of $G$. For this we will use results of Navarro--Tiep on the extension of odd degree characters from $G$ to $\widetilde{G}$. Specifically we have the following slight refinement of results from \cite{navarro-tiep:2015:irreducible-representations-of-odd-degree}.
\end{pa}

\begin{lem}[Navarro--Tiep]\label{prop:extension-to-GLn}
Assume $\chi \in \Irr(G)$ is an odd degree character and $\widetilde{\chi} \in \Irr(\widetilde{G} |\chi)$ covers $\chi$. Then one of the following holds:
\begin{enumerate}
	\item $\widetilde{\chi}(1) = \chi(1)$,

	\item $\widetilde{\chi}(1) = 2\chi(1)$ and $n = 2^r$ for some $r \geqslant 1$.
\end{enumerate}
\end{lem}

\begin{proof}
The case $n=2$ is easily checked so we may assume that $n > 2$. Let us assume $\widetilde{\chi} \in \mathcal{E}(\widetilde{G},\widetilde{s})$ then according to \cite[Lemma 4.5, Lemma 4.6]{navarro-tiep:2015:irreducible-representations-of-odd-degree} we have either (a) holds or the following holds
\begin{itemize}
	\item $\widetilde{\chi}(1) = 2\chi(1)$ and $C_{\widetilde{G}}(\widetilde{s}) \cong \GL_m^{\varepsilon}(q)^2$ with $n = 2m$.
\end{itemize}
By Lusztig's Jordan decomposition of characters we see that the index $[\widetilde{G} : C_{\widetilde{G}}(\widetilde{s})]_{p'}$ divides $\widetilde{\chi}(1)$, see \cite[Remark 13.24]{digne-michel:1991:representations-of-finite-groups-of-lie-type}, so $[\widetilde{G} : C_{\widetilde{G}}(\widetilde{s})]_2$ divides $\widetilde{\chi}(1)$ because $p$ is odd. If $m = 2^{r_1} + \cdots + 2^{r_t}$ is the $2$-adic expansion of $m$ then we have $[\widetilde{G} : C_{\widetilde{G}}(\widetilde{s})]_2 = 2^t$ by \cref{cor:index-GL}. However $\chi(1)$ is odd so we must have $t = 1$, which proves the statement.
\end{proof}

\begin{pa}
By \cref{prop:GLn-sigma-fixed} we know the effect of $\sigma$ on the irreducible characters of $\widetilde{G}$. Hence when an odd degree character of $G$ extends to $\widetilde{G}$ we can easily determine the effect of $\sigma$ on such a character. Thus we are left with considering the second case of \cref{prop:extension-to-GLn}.  For this case, we record the following observations.
\end{pa}

\begin{lemma}\label{lem:classinvariant}
Assume $\chi \in \Irr(G)$ is an irreducible character and $\widetilde{\chi} \in \Irr(\widetilde{G}| \chi)$ covers $\chi$. If the $G$-conjugacy class of $g \in G$ is invariant under conjugation by $\widetilde{G}$, then $\chi(g) = \widetilde{\chi}(1)\widetilde{\chi}(g)/\chi(1)$. In particular, we have $\chi(g)^{\sigma} = \chi(g)$ if and only if $\widetilde{\chi}(g)^{\sigma} = \widetilde{\chi}(g)$.
\end{lemma}

\begin{proof}
This follows immediately from the fact that $\Res_G^{\widetilde{G}}(\widetilde{\chi})$ is multiplicity free.
\end{proof}

\begin{lemma}\label{lem:unipsquare}
Recall our assumption that $p$ is odd and let $\chi\in\irr(G)$ be an irreducible character.  Then $\chi(u)^\sigma = \chi(u^2)$ for any unipotent element $u \in G$.
\end{lemma}

\begin{proof}
Let $\mathfrak{X}$ be a complex representation affording $\chi$.  Then $\chi(u)$ is the sum $\sum_{i=1}^{\chi(1)}\lambda_i$ of eigenvalues $\lambda_i$ of the matrix $\mathfrak{X}(u)$ and $\chi(u^2)$ is the sum $\sum_{i=1}^{\chi(1)}\lambda_i^2$ of eigenvalues of the matrix $\mathfrak{X}(u)^2$.  Hence, since $u$ is a $2'$-element, each $\lambda_i$ is a $2'$-root of unity, so $\chi(u)^\sigma=\sum_{i=1}^{\chi(1)} \lambda_i^2=\chi(u^2)$.
\end{proof}

\begin{pa}\label{pa:levi-id}
As we will see below, the case when $n=4$, i.e., when $G = \SL_4^{\varepsilon}(q)$, will need to be treated separately with ad-hoc methods. In particular, we will need some knowledge of the Levi subgroup
\begin{equation}\label{eq:levi}
\bL = \{\diag(A,B) \mid A,B \in \GL_2(\mathbb{K}) \text{ and }\det(B) = \det(A)^{-1}\} \leqslant \bG = \SL_4(\mathbb{K}).
\end{equation}
Note this subgroup is stable by the Frobenius endomorphism $F$. Let $W = N_{\bG}(\bT_0)/\bT_0$ be the Weyl group of $\bG$ with respect to $\bT_0$ and let $W_{\bL} = N_{\bL}(\bT_0)/\bT_0$ be the corresponding parabolic subgroup determined by $\bL$.

The section $N_{\bG}(\bL)/\bL$ of $\bG$ is isomorphic to the section $N_W(W_{\bL})/W_{\bL}$ of $W$, which has order $2$. Identifying $W$ with $\mathfrak{S}_4$ in the usual way, we have the non-trivial coset of $N_W(W_{\bL})/W_{\bL}$ is represented by the permutation $(1,3)(2,4)$. Let $n \in N_{\bG}(\bT_0)$ be the permutation matrix representing $(1,3)(2,4)$; note this matrix has determinant $1$. If $\imath_n : \bL \to \bL$ denotes the conjugation map defined by $\imath_n(l) = nln^{-1}$ then the map $\imath_nF : \bL \to \bL$ is a Frobenius endomorphism of $\bL$ stabilising $\bT_0$. Now assume $\bM = {}^g\bL$ is an $F$-stable $\bG$-conjugate of $\bL$. After possibly replacing $g$ by $gl$, for some $l \in \bL$, we may assume that conjugation by $g$ identifies the pair $(\bM,F)$ with either $(\bL,F)$ or $(\bL,\imath_nF)$. With this we are ready to prove the following lemmas.
\end{pa}

\begin{lem}\label{lem:L-u-u2-conj}
Assume $n = 4$ so that $G = \SL_4^{\varepsilon}(q)$ and recall that $p$ is odd. If $\bM = {}^g\bL$ is an $F$-stable $\bG$-conjugate of $\bL$, then any rational unipotent element $u \in \bM^F$ is $\bM^F$-conjugate to $u^2$.
\end{lem}

\begin{proof}
By \cref{pa:levi-id} each pair $(\bM,F)$ can be identified with the pair $(\bL,F')$ where $F'$ denotes either $F$ or $nF$, hence it suffices to prove the statement for the pair $(\bL,F')$. The unipotent conjugacy classes of $\bL$ are parameterised by the Jordan normal form. Let $\mathcal{O} \subseteq \bL$ be a unipotent conjugacy class. Then $\mathcal{O}$ is invariant under the map $x \mapsto x^2$ because the elements have the same Jordan normal form.

Assume now that $\mathcal{O}$ is $F'$-stable and $u \in \mathcal{O}^{F'}$. As the component group $C_{\bL}(u)/C_{\bL}^{\circ}(u)$ has order at most $|Z(\bL)/Z^{\circ}(\bL)|=2$ we have either $\mathcal{O}^{F'}$ is a single $\bL^{F'}$-conjugacy class or it's a union of two such classes, see \cite[4.3.5]{geck:2003:intro-to-algebraic-geometry}. Therefore, it suffices to show that for one element $u \in \mathcal{O}^{F'}$ we have $u$ is $\bL^{F'}$-conjugate to $u^2$. Applying \cite[4.3.5]{geck:2003:intro-to-algebraic-geometry} it suffices to find an element $t \in \bT_0$ such that ${}^tu = u^2$ and $t^{-1}F(t) \in C_{\bT_0}^{\circ}(u) \leqslant C_{\bL}^{\circ}(u)$.

Let $J = \left[\begin{smallmatrix} 1 & 1\\0 & 1 \end{smallmatrix}\right]$ and let $u$ be one of the elements $\diag(J,J)$, $\diag(J,\mathrm{I}_2)$, $\diag(\mathrm{I}_2,J)$, or $\diag(\mathrm{I}_2,\mathrm{I}_2)$. These elements represent the unipotent conjugacy classes of $\bL$. Setting $t = \diag(2a,a,a^{-1},2^{-1}a^{-1})$, for some $a \in \mathbb{G}_m$, one easily checks that ${}^tu = u^2$ and $t^{-1}F'(t) \in C_{\bT_0}^{\circ}(u)$ as desired.
\end{proof}

\begin{lem}\label{lem:SL4-GGGRs}
Assume $n=4$ so that $G = \SL_4^{\varepsilon}(q)$ and recall that $p$ is odd. If $u \in G$ is a non-regular unipotent element then $u$ is $G$-conjugate to $u^2$. In particular, we have $\Gamma_u^{\sigma} = \Gamma_u$ for any non-regular unipotent element $u \in G$.
\end{lem}

\begin{proof}
If $u \in \bG$ is a unipotent element then $u$ and $u^2$ have the same Jordan normal form so they are $\bG$-conjugate. Each unipotent element has a connected centraliser unless $u$ is either regular or conjugate to $\diag(J,J)$ with $J = \left[\begin{smallmatrix} 1 & 1\\0 & 1 \end{smallmatrix}\right]$. The arguments used in the proof of \cref{lem:L-u-u2-conj} show the first statement. The last statement follows from \cref{prop:galois-aut-GGGR}.
\end{proof}

\begin{rem}
As $2$ is a generator of $\mathbb{F}_p^{\times}$ we see from the proof of \cite[6.7]{tiep-zalesski:2004:unipotent-elements} that if $u \in G = \SL_4^{\varepsilon}(q)$ is regular unipotent then we need not necessarily have $u$ is $G$-conjugate to $u^2$.
\end{rem}

\begin{pa}
To understand the effect of $\sigma$ on the odd degree characters of $G$ we will need to be able to distinguish between odd degree characters which are contained in the same $\widetilde{G}$-orbit. To do this we will use the GGGRs of $G$, see \cref{sec:GGGRs}. In this direction we will need the following consequence of \cite{tiep-zalesski:2004:unipotent-elements}.
\end{pa}

\begin{proposition}\label{prop:GGGRvalues}
Let $\Gamma_u$ be a GGGR of $G = \SL_n^{\varepsilon}(q)$. Then the following hold.
\begin{enumerate}
	\item For any $g \in G$ we have $\Gamma_u(g) \in \Q(\sqrt{\eta p})$, where $\eta\in\{\pm1\}$ is such that $p\equiv\eta\pmod{4}$,
\item if $q$ is a square, $n$ is odd, or $n/(n,q-\epsilon)$ is even then $\Gamma_u(g) \in \mathbb{Z}$ for all $g \in G$.
\end{enumerate}
In particular, if $q\equiv \pm1\pmod 8$, $n$ is odd, or $n/(n,q-\epsilon)$ is even, then $\Gamma^\sigma=\Gamma$.
\end{proposition}

\begin{proof}
This follows from \cite[Theorem 1.8, Lemma 2.6, and Theorem 10.10]{tiep-zalesski:2004:unipotent-elements}, together with the fact that $\Gamma$ is a unipotently supported character of $G$.  The last statement follows by noting that $\sqrt{p}^\sigma=\sqrt{p}$ if $p\equiv \pm1\pmod 8$ and $\sqrt{p}^\sigma=-\sqrt{p}$ if $p\equiv \pm3\pmod 8$, since if $q$ is not a square, then $q\equiv p\pmod 8$. 
\end{proof}

\begin{pa}
We are now in a position to prove the second part of \cref{thm:SN2SgoodTypeA}, thus concluding its proof. Namely, we need to show that the simple groups $\PSL_n^{\varepsilon}(q)$ are SN2S-Good. Note that \cref{prop:conjIF-p=2,prop:p-2-conjFI} show that $\PSL_n^{\varepsilon}(q)$ is SN2S-Good when $q$ is even so our standing assumption that $q$ is odd is not restrictive. As we have already shown in \cref{prop:typeAconjFI} that \cref{cond:conjFI} holds for $\PSL_n^{\varepsilon}(q)$, we need only show that \cref{cond:conjIF} holds for the corresponding quasisimple groups. Parts of the argument are similar to that used in \cite[Theorem 4.15, part (3)]{schaeffer-fry:2015:odd-degree-characters} but we include it here for completeness.
\end{pa}

\begin{prop}\label{thm:notpm3mod8}
Assume $G$ is perfect so the quotient $G/Z = \PSL_n^\epsilon(q)$ is simple. Then any quasisimple group whose simple quotient is isomorphic to $\PSL_n^\epsilon(q)$ satisfies \cref{cond:conjIF}.
\end{prop}

\begin{proof}
We may assume that $G$ has a trivial Schur multiplier because the case $\PSL_2(9)\cong \mathfrak{A}_6$ was treated in \cite{schaeffer-fry:2015:odd-degree-characters}. In this case $G$ is a Schur cover of $S$ and it suffices to show that $G$ satisfies \cref{cond:conjIF}, c.f., \cref{pa:conditions-remarks}.

Let $Q$ and $\chi\in\irr_{2'}(G)$ be as in the hypothesis of \cref{cond:conjIF}. By \cref{prop:seriesabove} there exists a $\sigma$-fixed irreducible character $\widetilde{\chi} \in \Irr(\widetilde{G}|\chi)$ covering $\chi$. A well known result of Kawanaka assures that there exists a unipotent element $u \in G \subseteq \widetilde{G}$ whose corresponding GGGR $\widetilde{\Gamma}_u$ of $\widetilde{G}$ satisfies $\langle \widetilde{\Gamma}_u, \widetilde{\chi}\rangle_{\widetilde{G}} = 1$, see \cite[3.2.18]{kawanaka:1985:GGGRs-and-ennola-duality} or \cite[15.7]{taylor:2016:GGGRs-small-characteristics}. By the construction of the GGGRs we have $\widetilde{\Gamma}_u = \Ind_G^{\widetilde{G}}(\Gamma_u)$ where $\Gamma_u$ is the GGGR of $G$ determined by $u$. Hence, applying Frobenius reciprocity we have the restriction $\Res_G^{\widetilde{G}}(\widetilde{\chi})$ contains a unique irreducible constituent $\chi_0 \in \Irr(G)$ satisfying $\langle \Gamma_u, \chi_0 \rangle_G = 1$. Assume that $\Gamma_u^{\sigma} = \Gamma_u$. Then as $\widetilde{\chi}^{\sigma} = \widetilde{\chi}$ we have $\chi_0^\sigma$ is also a constituent of $\Res_G^{\widetilde{G}}(\widetilde{\chi})$ satisfying $\langle \Gamma_u, \chi_0^\sigma\rangle_G=1$. The uniqueness of such a character forces $\chi_0^{\sigma} = \chi_0$. By Clifford theory, we may write $\chi=\chi_0^g$ for some $g\in \widetilde{G}$, so $\chi^{\sigma} = (\chi_0^{\sigma})^g = \chi$.

If either $q\equiv \pm1\pmod 8$, or $n/(n,q-\epsilon)$ is even, then by \cref{prop:GGGRvalues} we have each GGGR $\Gamma_u$ of $G$ is $\sigma$-fixed so the above argument applies and \cref{cond:conjIF} holds for $G$. Thus we may assume that $q\equiv\pm3 \pmod{8}$ and $n/(n,q-\epsilon)$ is odd. If $\chi$ extends to $\widetilde{G}$ then Gallagher's theorem implies that $\Res_G^{\widetilde{G}}(\widetilde{\chi}) = \chi$ so $\chi^{\sigma} = \chi$. We may therefore assume that $\chi$ does not extend to $\widetilde{G}$, so by \cref{prop:extension-to-GLn} we must have $n = 2^r$ for some $r \geqslant 1$. Now, as $n/(n,q-\epsilon)$ is odd, we must have $n = 2^r$ divides the $2$-part $(q-\varepsilon)_2$ of $q-\varepsilon$. But $q - \varepsilon \equiv \pm 3 - \varepsilon \pmod{8}$, which is either $\pm 2 \pmod{8}$ or $4 \pmod{8}$. Hence $(q-\varepsilon)_2$ is either $2$ or $4$ so $n$ is either $2$ or $4$.

The case $n=2$ is treated in \cite{schaeffer-fry:2015:odd-degree-characters} so we need only show that \cref{cond:conjIF} holds for $G = \SL_4^\epsilon(q)$ with $q\equiv\pm 3\pmod8$. By \cref{lem:SL4-GGGRs} we have $\Gamma_u^{\sigma} = \Gamma_u$ unless $u$ is regular unipotent, so the above argument shows that $\chi^{\sigma} = \chi$ unless $\chi$ is a regular character. Assume the $\sigma$-invariant character $\widetilde{\chi} \in \Irr(\widetilde{G} | \chi)$ covering $\chi$ is contained in the Lusztig series $\mathcal{E}(\widetilde{G},\widetilde{s})$. Then by \cref{prop:seriesabove} we may assume $\widetilde{s}$ is of $2$-power order; in particular $s$ is of $2$-power order.

We now aim to show that $\chi(g)^{\sigma} = \chi(g)$ for each $g \in G$, thus showing $\chi^{\sigma} = \chi$. First, assume $g$ is semisimple.  Then as $\bG$ is simply connected, we have $C_{\bG}(g)$ is connected. This easily implies that the $G$-conjugacy class containing $g$ is invariant under conjugation by $\widetilde{G}$ so $\chi(g)^{\sigma} = \chi(g)$ in this case by \cref{lem:classinvariant}. Next, assume $g$ is unipotent, so by \cref{lem:unipsquare} we have $\chi(g)^{\sigma} = \chi(g^2)$. If $g$ is not a regular unipotent element then $g$ and $g^2$ are $G$-conjugate, c.f., \cref{lem:SL4-GGGRs}, so again $\chi(g)^{\sigma} = \chi(g)$.

If $g$ is regular unipotent then we claim $\chi(g) = 0$, thus trivially $\chi(g)^{\sigma} = \chi(g)$. By \cite[Corollary 14.38]{digne-michel:1991:representations-of-finite-groups-of-lie-type} we have $\chi(g) = 0$ if $D_G(\chi)$ does not occur as a constituent of any Gelfand--Graev character. Assume for a contradiction that $D_G(\chi)$ does occur in some Gelfand--Graev character. Then $\chi$ is both regular and semisimple. This implies $\widetilde{\chi}$ is both regular and semisimple. However, by \cite[15.6, 15.10]{bonnafe:2006:sln} this can only happen if the trivial and sign character of the Weyl group of $C_{\widetilde{\bG}}(s)$ coincide. Clearly this is not the case, so we must have $\chi(g) = 0$ as desired.

We now need only consider the case where $g = g_sg_u = g_ug_s$ with $g_s\neq1$ semisimple and $g_u\neq 1$ unipotent. Note that we have $C_{\bG}(g) = C_{C_{\bG}(g_s)}(g_u)$ and the centraliser $C_{\bG}(g_s)$ is a Levi subgroup of $\bG$. The subgroup $C_{\bG}(g_s)$ is $\bG$-conjugate to a standard Levi subgroup of $\bG$ so $C_{\bG}(g_s)$ is isomorphic to either $\GL_3(\mathbb{K})$, $\GL_2(\mathbb{K}) \times \mathbb{G}_m$, or the subgroup $\bL$ defined in \cref{eq:levi}. In the first two cases the centraliser of every unipotent element is connected, which implies $C_{\bG}(g)$ is connected. As argued above we can conclude from \cref{lem:classinvariant} that $\chi(g)^{\sigma} = \chi(g)$.

Thus we are left with the case where $C_{\bG}(g_s)$ is $\bG$-conjugate to $\bL$. As is remarked in \cite[\S25.A]{bonnafe:2006:sln} we have $\chi(g) = {}^*R_{C_{\bG}(g_s)}^{\bG}(\chi)(g)$ so we need only show that ${}^*R_{C_{\bG}(g_s)}^{\bG}(\chi)(g)^{\sigma} = {}^*R_{C_{\bG}(g_s)}^{\bG}(\chi)(g)$. The class function ${}^*R_{C_{\bG}(g_s)}^{\bG}(\chi)$ is a $\mathbb{Z}$-linear combination of irreducible characters, so it suffices to show that $\lambda(g)^{\sigma} = \lambda(g)$ for each irreducible constituent $\lambda$ of ${}^*R_{C_{\bG}(g_s)}^{\bG}(\chi)$.

Assume $\lambda$ is such a constituent. Then $\lambda$ is a contained in a Lusztig series of $C_G(g_s)$ labelled by a semisimple element which is $G^{\star}$-conjugate to $s$, see \cite{bonnafe:2006:sln}. As mentioned above, we have $s$ is of $2$-power order, hence so is any $G^{\star}$-conjugate of $s$. By \cref{lem:galoisLseries} we thus have the Lusztig series containing $\lambda$ is $\sigma$-invariant. If $\omega_{\lambda} : Z(C_G(g_s)) \to \Ql^{\times}$ is the central character of $\lambda$, then $\lambda(g) = \omega_{\lambda}(g_s)\lambda(g_u)$ because $g_s \in Z(C_G(g_s))$. As the Lusztig series containing $\lambda$ is $\sigma$-invariant, we have $\omega_{\lambda}(g_s)^{\sigma} = \omega_{\lambda}(g_s)$ by \cref{lem:central-char}. Applying \cref{lem:unipsquare}, we see that $\lambda(g_u)^{\sigma} = \lambda(g_u^2)$.  However, by \cref{lem:L-u-u2-conj} we must have $\lambda(g_u^2) = \lambda(g_u)$ because $C_{\bG}(g_s)$ is an $F$-stable $\bG$-conjugate of $\bL$. In particular, we have $\lambda(g)^{\sigma} = \omega_{\lambda}(g_s)^{\sigma}\lambda(g_u)^{\sigma} = \omega_{\lambda}(g_s)\lambda(g_u) = \lambda(g)$ as desired.
\end{proof}

\section*{Acknowledgements}
The authors would like to thank the organizers of the 2015 workshop on ``Representations of Finite Groups" at Mathematisches Forschungsinstitut Oberwolfach, where this collaboration began.  The first-named author was supported in part by a grant from the Simons Foundation (Award \#351233) and by an NSF-AWM Mentoring Travel Grant. In addition, she would like to thank G.\ Malle and B.\ Sp{\"a}th for their hospitality and helpful discussions during the research visit to TU Kaiserslautern supported by the latter grant. The second author gratefully acknowledges the financial support of an INdAM Marie-Curie Fellowship and grants CPDA125818/12 and 60A01-4222/15 of the University of Padova. He also wishes to thank the MSU Denver for its kind hospitality during a two-week long visit.

\setstretch{0.96}
\renewcommand*{\bibfont}{\small}
\printbibliography
\end{document}